\documentclass[11pt, reqno]{amsart}
\usepackage{multirow}
\usepackage{tikz}
\usepackage{caption}
\usetikzlibrary{
  decorations.pathmorphing,%
  decorations.markings,%
  decorations.pathreplacing,%
  decorations.text,%
  calc,%
  patterns,%
  shapes.geometric,%
  arrows,
  decorations.shapes,
  positioning,
  plotmarks,
  shadings,
  math,
  intersections
}

\definecolor{jeffColor}{RGB}{102, 0, 204}
\definecolor{yaizaColor}{RGB}{0, 153, 153}

\definecolor{periodColor}{RGB}{255, 167, 105}
\definecolor{dark-green}{RGB}{135, 194, 130}
\tikzset{>=latex} 
\tikzset{font=\small}
\tikzset{mark size=1.5pt, mark options=thin}
\tikzset{pin distance=4pt,
  every pin edge/.style={<-, thin, shorten <= -2pt}}
\usepackage{array,booktabs,tabularx}
\usepackage{graphicx}
\usepackage{amssymb}
\usepackage{enumerate}
\usepackage{color}
\usepackage{hyperref}
\usepackage{esint}
\usepackage{mathtools}
\usepackage{dsfont}
\usepackage{ esint }
\usepackage{placeins}
\usepackage[show]{ed}
\usepackage[final]{showkeys}
\newcommand{\N}[1]{{\left\vert\kern-0.25ex\left\vert\kern-0.25ex\left\vert #1 
    \right\vert\kern-0.25ex\right\vert\kern-0.25ex\right\vert}}

\usepackage{comment}
\usepackage{enumitem}
\usepackage{soul}
\usepackage{needspace}
\usepackage[normalem]{ulem}

\usepackage{multicol}
\usepackage{tabto}

\newcommand{\Ell}{\operatorname{ell_h}}
\newtheorem{lemma}{Lemma}
\newtheorem{theorem}{Theorem}

\newtheorem{corollary}[lemma]{Corollary}

\numberwithin{theorem}{section}

\theoremstyle{definition}

\newtheorem{remark}[lemma]{Remark}

\usepackage{dsfont}
\newcommand{\1}{\mathds{1}}

\newcommand{\semi}{\mathcal{N}}

\newcommand{\Id}{\operatorname{Id}}
\newcommand{\mc}[1]{\mathcal{#1}}

\newcommand{\comp}{\operatorname{comp}}

\newcommand{\scatPhase}{^{\textup{sc}}\overline{T^*\mathbb{R}}}

\def\XXint#1#2#3{{\setbox0=\hbox{$#1{#2#3}{\int}$} \vcenter{\hbox{$#2#3$}}\kern-.5\wd0}}

\DeclareMathOperator{\supp}{supp}

\newcommand{\e}{\varepsilon}

\newcommand{\Oph}{Op_h}
\numberwithin{equation}{section}
\numberwithin{lemma}{section}

\DeclareMathOperator{\Ells}{\operatorname{ell_h^{sc}}}
\DeclareMathOperator{\WFs}{\operatorname{WF_{h}^{sc}}}
\DeclareMathOperator{\WF}{\operatorname{WF_{h}}}

\newcommand{\ad}{\operatorname{ad}}
\newcommand{\sgn}{\operatorname{sgn}}
\renewcommand{\Im}{\operatorname{Im}}

\title[Asymptotic expansions of the spectral function in dimension one]
{Complete asymptotic expansions of the spectral function for symbolic perturbations of almost periodic Schr\"odinger operators in dimension one}

\author[J. Galkowski]{Jeffrey Galkowski}

\begin{document}
\begin{abstract}
In this article we consider asymptotics for the spectral function of Schr\"odinger operators on the real line. Let $P:L^2(\mathbb{R})\to L^2(\mathbb{R})$ have the form
$$
P:=-\tfrac{d^2}{dx^2}+W,
$$
where $W$ is a self-adjoint first order differential operator with certain modified almost periodic structure. We show that the kernel of the spectral projector, $\1_{(-\infty,\lambda^2]}(P)$ has a full asymptotic expansion in powers of $\lambda$. In particular, our class of potentials $W$  is stable under perturbation by formally self-adjoint first order differential operators with smooth, compactly supported coefficients. Moreover, it includes certain potentials with \emph{dense pure point spectrum}. The proof combines the gauge transform methods of Parnovski-Shterenberg and Sobolev with Melrose's scattering calculus. 
\end{abstract}
\vspace*{-1.2cm}
\maketitle

\vspace*{-1cm}



\section{Introduction}
Let
$$
P:=D_x^2+W_1D_x+D_xW_1+W_0:L^2(\mathbb{R})\to L^2(\mathbb{R}),
$$
where $W_j\in C^\infty(\mathbb{R};\mathbb{R})$. We study the spectral projection for $P$, $\1_{(-\infty,\lambda^2]}(P)$, when $W_j$, $j=0,1$ satisfy certain almost periodic conditions. Denote by $e_\lambda(x,y)$ the kernel of $\1_{(-\infty,\lambda^2]}(P)$.

We assume that there is $\Theta\subset \mathbb{R}$ countable such that $-\Theta=\Theta$, $0\in \Theta$, and for all $k,N\geq0$ there is $C_{k,N}>0$ such that
\begin{equation}
\label{e:potentialForm}
W_j(x)= \sum_{\theta\in \Theta} e^{i\theta x}w_{\theta,j}(x),\qquad 
|\partial_x^kw_{\theta,j}(x)|\leq C_{k,N}\langle x\rangle^{-k}\langle \theta\rangle^{-N}.
\end{equation}
Before stating the general conditions on $w_\theta$ (see \S\ref{s:almostPeriodic}), we give two consequences of our main theorem (Theorem~\ref{t:main}). Let ${\omega}:=(\omega_1,\dots \omega_d)\in\mathbb{R}^d$. We say ${\omega}$ satisfies the diophantine condition if there are $c,\mu>0$ such that 
\begin{equation}
\label{e:diophantine}
|{\bf{n}}\cdot {\omega}|>c |{\bf{n}}|^{-\mu}, \qquad {\bf{n}}\in \mathbb{Z}^d\setminus\{0\}.
\end{equation}

\begin{theorem}
\label{t:almostPeriodic}
Suppose ${\omega}\in \mathbb{R}^d$ satisfies the diophantine condition~\eqref{e:diophantine} and $W_j$ are as in~\eqref{e:potentialForm} with $\Theta=\mathbb{Z}^d\cdot {\omega}$ and for all $k,N\geq0$ there is $C_{k,N}>0$ such that
$$
|\partial_x^kw_{{\bf{n}}\cdot \omega,j}(x)|\leq C_{k,N}\langle x\rangle^{-k}\langle{\bf{ n}}\rangle^{-N},\qquad {\bf{n}}\in \mathbb{Z}^d,\,j=0,1,
$$ then for $|x-y|>c$,
\begin{equation}
\label{e:asymptotic}
\begin{gathered}
e_\lambda(x,y)\sim \cos(\lambda (x-y))\sum_j\lambda^{-j}a_j(x,y)+\sin(\lambda(x-y))\sum_j \lambda^{-j}b_j(x,y),\qquad e_\lambda(x,x)\sim \sum_j \tilde{a}_j \lambda^{j+1}
\end{gathered}
\end{equation}
where $a_0=0$ and $b_0=\frac{2}{\pi(x-y)}$.  Moreover, we have an oscillatory integral expression for $e_{\lambda}(x,y)$ valid uniformly for $(x,y)$ in any compact subset of $\mathbb{R}^2$.
\end{theorem}
\begin{remark}
It is easy to see that the condition~\eqref{e:diophantine} is generic in the sense that it is satisfied for Lebesgue almost every $\omega \in [-1,1]^d$.
\end{remark}

Next, we state a theorem in the limit periodic case. 

\begin{theorem}
\label{t:limitPeriodic}
Let $\{m_n\}_{n=1}^\infty \subset \mathbb{Z}_+$, and $\Theta=\Theta_+\cup -\Theta_+\cup\{0\}$ where $\Theta_+=\{\theta_n\}_{n=1}^\infty,$ $\theta_n:=m_n/n$. Suppose that  $W_j$ are as in~\eqref{e:potentialForm} such that for all $k,N\geq 0$ there is $C_{k,N}>0$ such that
$$|\partial_x^kw_{\theta_n,j}(x)|\leq C_{k,N}\langle x\rangle^{-k}\langle n\rangle^{-N},\qquad n\geq 1,\,j=0,1,$$ 
then~\eqref{e:asymptotic} holds.
\end{theorem}

In both Theorems~\ref{t:almostPeriodic} and~\ref{t:limitPeriodic}, one may add \emph{any} formally self-adjoint first order differential operator $W_{sym}=a_1(x)D_x+b_1(x)$  whose coefficients satisfy $|\partial_x^ka_i(x)|\leq C_k\langle x\rangle^{-k}$ to $W$ and $W+W_{sym}$ will satisfy the assumptions of the Theorem. In addition, Theorems~\ref{t:almostPeriodic} and~\ref{t:limitPeriodic} include examples with arbitrarily large embedded eigenvalues and Theorem~\ref{t:limitPeriodic} includes examples with dense pure point spectrum (See Appendix~\ref{a:example}). 

While full asymptotic expansions are known in the the case that $W$ is compactly supported~\cite{PoSh:83,Va:84} and in the case that $W_1=0$, $W_0=\sum_\theta e^{i\theta x}v_\theta $ with $v_\theta \in \mathbb{C}$ and $\Theta$ satisfying the assumptions of Theorem~\ref{t:almostPeriodic}~\cite{PaSh:16}, to the author's knowledge Theorems~\ref{t:almostPeriodic} and~\ref{t:limitPeriodic} are the first to allow for both types of behavior. The work~\cite{PaSh:16} followed the approach developed in~\cite{PaSh:12,PaSh:09} for the study of the integrated density of states a subject which, for periodic Schro\"odinger operators, has been the focus of a long line of articles (see e.g.~\cite{So:05,So:06,Ka:00,HeMo:98}).

\subsection{Discussion of the proof}

We choose not to state our general results until all of the necessary preliminaries have been introduced (see Theorem~\ref{t:main}). Instead, we outline how our proof draws on and differs from the work of Parnovski--Shterenberg~\cite{PaSh:09,PaSh:12,PaSh:16} and Morozov--Parnovski--Shterenberg~\cite{MoPaSh:14}. These papers handle the much more difficult higher dimensional case of the above problem when $W(x,D)$ is replaced by a potential $V(x)=\sum_{\theta\in \Theta}v_\theta e^{i\theta x}$ where $v_\theta\in\mathbb{C}$ and $\Theta$ is assumed to be countable and satisfying certain diophantine conditions.  The crucial technique used in those articles is the gauge transform (developed in~\cite{So:05,So:06,PaSo:10}) i.e. conjugating the operator $P$ by $e^{iG}$ for some pseudodifferential $G$ constructed so that the conjugated operator takes the form
$H_0+R$ where $H_0$ is a constant coefficient differential operator near frequencies $|\xi|\sim \lambda$ and away from certain resonant zones in the Fourier variable and where $R=O(\lambda^{-N})_{H^{-N}\to H^N}$. The authors are then able to make a sophisticated analysis of the operator $H_0$ acting on Besicovitch spaces. This analysis uses in a crucial way that $H_0$ acts nearly diagonally i.e. that the operator can be thought of as a direct sum of operators acting on resonant frequencies and is diagonal away from these frequencies. The authors write a more or less explicit, albeit complicated, integral formula for the spectral function and then directly analyze this integral.

In this article, we take a somewhat different approach to the second step of the above analysis. Namely, we start with our operator $P$ and, after conjugation by $e^{iG}$, are able to reduce to the case of $H_0+R$ where $H_0$ is a scattering pseudodifferential operator~\cite{Me:94} near the frequencies $|\xi|\sim \lambda$. However, because we have simplified our problem by working in one dimension, resonant zones do not occur. In particular, we will prove a limiting absorption principle for $H_0$ at high enough energies and show that the resulting resolvent operators $(H_0-\lambda^2\mp i0)^{-1}$ satisfy certain `semiclassical outgoing/incoming' properties. These, roughly speaking, state that the resolvent transports singularities in only one direction along the Hamiltonian flow for the symbol of $H_0$ and that these singularities do not return from infinity.  With this in hand, we are able understand the spectral projector for $H_0$ using the wave method of Levitan~\cite{Le:52}, Avakumovi\'c~\cite{Av:56} and H\"ormander~\cite{Ho:68} and hence, using an elementary spectral theory argument, to understand the spectral function for $P$. The crucial fact allowing the proof of a limiting absorption principle is that $H_0$ may be chosen such that the `non-scattering pseudodifferential' part is identically zero on frequencies near $\lambda$.

\medskip\noindent\textbf{Acknowledgements.} Thanks to Leonid Parnovski for many helpful discussions about the literature on almost periodic operators; especially for comments on the articles~\cite{PaSh:09,PaSh:12,PaSh:16} as well as for comments on an early draft. Thanks also the the anonymous referee for careful reading and helpful comments.

\section{General assumptions}
\subsection{Pseudodifferential classes}
We work with pseudodifferential operators in Melrose's scattering calculus~\cite{Me:94}. Since we are working in the simple setting of $\mathbb{R}$, we will not review the construction of an invariant calculus. Instead, we say that $a\in C^\infty(\mathbb{R}^{2})$ lies is $S^{m,n}$ if for all $\alpha,\beta\in \mathbb{N}.$
\begin{equation}
\label{e:scatteringSymbol}
|\partial_x^\alpha\partial_\xi^\beta a(x,\xi)|\leq C_{\alpha\beta}\langle x\rangle^{n-\alpha}\langle \xi\rangle^{m-\beta}.
\end{equation}
We define the seminorms on $S^{m,n}$ by
$$
\|a\|^{m,n}_{\beta,\alpha}=\sum_{j=0}^\alpha\sum_{k=0}^\beta \sup |\partial_x^j\partial_\xi^\beta a(x,\xi)\langle x\rangle^{-n+j}\langle \xi\rangle^{-m+k}|
$$
When it is convenient, we will say $\semi=(m,n,\alpha,\beta)\subset \mathbb{N}^{4}$ is a choice of a seminorm on $S^{m,n}$.

It will also be convenient to have the standard symbol classes on $\mathbb{R}$. For this, we say $a\in C^\infty(\mathbb{R}^2)$ lies in $S^{m}$ if 
$$
|\partial_x^\alpha\partial_\xi^\beta a(x,\xi)|\leq C_{\alpha\beta}\langle \xi\rangle^{m-\beta}.
$$
Note that $S^{m,n}\subset S^m$. We also define the corresponding classes of pesudodifferential operators:
$$
\Psi^{m,n}:=\{ a(x,hD)\mid a\in S^{m,n}\},\qquad \Psi^{m}:=\{a(x,hD)\mid a\in S^m\},
$$
where for $a\in S^m$, 
$$
a(x,hD)u:=\frac{1}{2\pi h}\int e^{\frac{i}{h}(x-y)\xi}a(x,\xi)u(y)dyd\xi.
$$
We sometimes write $\Oph(a)$ for the operator $a(x,hD)$.

Our pseudodifferential operators will have polyhomogeneous symbols. That is, they will be given by $a\in S^{m,n}$, $b\in S^m$ such that there are $a_j\in S^{m-j,n-j}$, $b_j\in S^{m-j}$ satisfying
$$
a(x,\xi)-\sum_{j=0}^{N-1}h^ja_j(x,\xi)\in h^NS^{m-N,n-N},\qquad b(x,\xi)-\sum_{j=0}^{N-1}h^jb_j(x,\xi)\in h^NS^{m-N}.
$$
We will abuse notation slightly from now and and write $a\in S^{m,n}$, $b\in S^m$ to mean that $a$ and $b$ have such expansions and $\Psi^{m,n}$, $\Psi^m$ for the corresponding operators.

Note that both $\Psi^{m,n}$ and $\Psi^m$ come with well behaved symbol maps, $\sigma_{m,n}:\Psi^{m,n}\to S^{m,n}$ and $\sigma_m:\Psi^m\to S^m$ respectively such that 
\begin{gather*}
0\to hS^{m-1,n-1}\overset{a(x,hD)}{\longrightarrow} \Psi^{m,n}\overset{\sigma_{m,n}}{\longrightarrow}S^{m,n}\to 0,\qquad 0\to hS^{m-1}\overset{a(x,hD)}{\longrightarrow} \Psi^{m}\overset{\sigma_{m}}{\longrightarrow}S^{m}\to 0.
\end{gather*}
are short exact sequences. Moreover,
$$
\sigma_{m_1+m_2,n_1+n_2}(AB)=\sigma_{m_1,n_1}(A)\sigma_{m_2,n_2}(B),\qquad \sigma_{m}(AB)=\sigma_{m_1}(A)\sigma_{m_2}(B),
$$
and
\begin{equation*}
\begin{gathered}
\sigma_{m_1+m_2-1,n_1+n_2-1}(ih^{-1}[A,B])=\{\sigma_{m_1,n_1}(A),\sigma_{m_2,n_2}(B)\},\\ \sigma_{m_1+m_2-1}(ih^{-1}[A,B])=\{\sigma_{m_1}(A),\sigma_{m_2}(B)\},
\end{gathered}\qquad \{a,b\}:=\partial_\xi a\partial_x b-\partial_\xi b\partial_x a.
\end{equation*}

For future use, we define norms as follows,
$$
\|u\|_{H_h^{s_1,s_2}}:=\|\langle x\rangle^{s_2}u\|_{H_h^{s_1}},\qquad \|u\|_{H_h^{s_1}}:=\|\langle -h^2\partial_x^2\rangle^{s_1/2} u\|_{L^2}.
$$
We recall the following estimates for pseudodifferential operators
\begin{lemma}
Let $a\in S^{m,n}$, $b\in S^m$. Then
$$
\|a(x,hD)u\|_{H_h^{s_1-m,s_1-n}}\leq C_a \|u\|_{H_h^{s_1,s_2}},\qquad \|b(x,hD)u\|_{H_h^{s_1-m,s_2}}\leq C_b \|u\|_{H_h^{s_1,s_2}},
$$
The maps $S^{m,n}\overset{a(x,hD)}{\longrightarrow} \mc{L}(H_h^{s_1,s_2},H_h^{s_1-m,s_2-n})$ and  $S^{m}\overset{b(x,hD)}{\longrightarrow} \mc{L}(H_h^{s_1},H_h^{s_1-m})$ are continuous.
\end{lemma}

In preparation for the gauge transform method, we prove two  preliminary lemmas on exponentials of elements of $\Psi^0$.
\begin{lemma}
\label{l:eiG}
Let $G\in \Psi^0$ self-adjoint. Then $e^{iG}\in \Psi^0$.
\end{lemma}
\begin{proof}
Let $g\in S^0$ such that $G=\Oph(g)$ and $A_0(t):=\Oph(e^{itg})$. We compute
$$
D_t(e^{-itG}A_0(t))=e^{-itG}(-GA_0+\Oph(ge^{itg}))=e^{-itG}h\Oph(r_1(t))
$$
where $r_1\in S^{-1}$. Now, suppose that we have $B_j(t)$, $j=1,\dots, N-1$, $B_j\in \Psi^{-j}$ such that with $A_{N-1}(t):=A_0(t)+\sum_{j=1}^{N-1}h^jB_j(t)$, 
$$
D_t(e^{-itG}A_N(t))=e^{-itG}h^N\Oph(r_N(t))
$$
with $r_N\in S^{-N}$. Then, putting $B_N(t)=\Oph(-i\int_0^te^{i(t-s)g}r_N(s)ds)$ we have
 \begin{align*} 
 D_t(e^{-itG}(A_N(t)+h^NB_N(t))&=e^{-itG}h^N\big(\Oph(r_N(t))-GB_N(t)+D_tB_N(t)\big)\\
 &=e^{-itG}h^{N+1}\Oph(r_{N+1}(t))
 \end{align*}
for some $r_{N+1}\in S^{-N-1}$. Putting $A\sim A_0+\sum_j h^jB_j(t)$, we have
$$
D_t(e^{-itG}A(t))=e^{-itG}O_t(h^\infty)_{\Psi^{-\infty}}.
$$
In particular, integrating, we have 
$$
e^{itG} =A(t)+ \int_0^te^{i(t-s)G}R_\infty(s)ds,\qquad R_\infty(s)=O(h^\infty)_{\Psi^{-\infty}}
$$
Therefore, since for all  $N$, $A(t):H_h^{-N}\to H_h^{-N}$ and $R_\infty:H_h^{-N}\to H_h^N$ are bounded, the fact that $e^{itG}:L^2\to L^2$ is bounded implies that for $N\geq 0$, $e^{itG}:H_h^{-N}\to H_h^{-N}$. But then for $u,v\in C_c^\infty$, 
$$
|\langle e^{itG}u,v\rangle_{L^2}|=|\langle u,e^{-itG}v\rangle_{L^2}|\leq \|u\|_{H_h^N}\|e^{-itG}v\|_{H_h^{-N}}\leq C\|u\|_{H_h^N}\|v\|_{H_h^{-N}}.
$$
In particular, by density, we have $e^{itG}:H_h^N\to H_h^N$ is bounded for all $N$ and hence 
$$
e^{itG}=A(t)+O(h^\infty)_{\Psi^{-\infty}}.
$$

 From the construction, it is clear that since $G$ is polyhomogeneous, so is $e^{itG}$.
\end{proof}

\begin{lemma}
\label{l:conjugate}
Let $G\in \Psi^0$ self adjoint, and $P\in \Psi^m$, 
$$
e^{iG}Pe^{-iG}=\sum_{k=0}^{N-1}\frac{i^k\ad_G^kP}{k!} +O(h^{N})_{H_h^s\to H_h^{s+N-m}}
$$
where 
$
\ad_AB=[A,B].
$
\end{lemma}
\begin{proof}
Note that 
$$
(D_t)^ke^{itG}Pe^{-itG}=e^{itG}\ad_G^kPe^{-itG}
$$
and in particular, 
$$
e^{itG}Pe^{-itG}=\sum_{k=0}^{N-1}\frac{t^ki^k}{k!}\ad_G^kP+\int_0^t\frac{(t-s)^{N-1}i^N}{(N-1)!}e^{isG}\ad_G^{N}Pe^{-isG}ds
$$
Now, $\ad_G^NP\in h^{N}\Psi^{m-N}$ and hence, the lemma follows by putting $t=1$ and recalling that $e^{ isG}\in \Psi^0$.
\end{proof}

\subsection{Ellipticity}

Next, we recall the notion of the elliptic set for elements of $\Psi^m$ and $\Psi^{m,n}$. To this end, we compactify $T^*\mathbb{R}$ in the fiber variables to $\overline{T^*\mathbb{R}}\cong \mathbb{R}\times [0,1]$ for $\Psi^m$ and in both the fiber and position variables to $\scatPhase\cong [-1,1]\times [-1,1]$ for $\Psi^{m,n}$. In particular, the boundary defining functions on $\scatPhase$ are $\pm x^{-1}$ near $\pm x=\infty$ and $\pm \xi^{-1}$ near $\pm \xi=\infty$ and those for $\overline{T^*\mathbb{R}}$ are $\pm x^{-1}$. We can now define the elliptic set of $A\in \Psi^{m,n}/\Psi^m$, $\Ells(A)\subset \scatPhase$, and $\Ell(A)\subset \overline{T^*\mathbb{R}}$ respectively as follows.  We say $\rho\in \Ells(A)$ if there is a neighborhood, $U\subset  \scatPhase$ of $\rho$ such that 
$$
\inf_{(x,\xi)\in U}\langle x\rangle^{-m}\langle \xi\rangle^{-n}|\sigma_{m,n}(A)(x,\xi)|>0.
$$
We say that $\rho \in \Ell(A)$ if there is a neighborhood, $U\subset  \overline{T^*\mathbb{R}}$ of $\rho$ such that 
$$
\inf_{(x,\xi)\in U}\langle \xi\rangle^{-n}|\sigma_{m,n}(A)(x,\xi)|>0.
$$

Next, we define the  wavefront set for an element of $\Psi^m$, $\WF(A)\subset \overline{T^*\mathbb{R}}$ and the scattering wavefront set of $A\in \Psi^{m,n}$, $\WFs(A)\subset \scatPhase$. For $A\in \Psi^{m}$, we say $\rho\notin \WF(A)$ if there is $B\in \Psi^{0}$ such that $\rho\in \Ell(B)$ and 
$$
\|BA\|_{H_h^{-N}\to H_h^{N}}\leq C_Nh^N.
$$
For $A\in \Psi^{m,n}$, we say $\rho\notin \WFs(A)$ if there is $B\in \Psi^{0,0}$ such that $\rho\in \Ells(B)$ and 
$$
\|BA\|_{H_h^{-N,-N}\to H_h^{N,N}}\leq C_Nh^N.
$$

We can now state the standard elliptic estimates.
\begin{lemma}
\label{l:elliptic}
Suppose $P\in \Psi^{m,n}$, $A\in \Psi^{0,0}$, with $\WFs(A)\subset \Ells(P)$. Then there is $C>0$ such that for all $N$ there is $C>0$ such that
$$
\|Au\|_{H_h^{s,k}}\leq C\|Pu\|_{H_h^{s-m,k-n}}+C_Nh^N\|u\|_{H_h^{-N,-N}}.
$$
If instead $P\in \Psi^{m}$, $A\in \Psi^{0}$, with $\WF(A)\subset \Ell(P)$. Then there is $C>0$ such that for all $N>0$ there is $C_N>0$ such that
$$
\|Au\|_{H_h^{s}}\leq C\|Pu\|_{H_h^{s-m}}+C_Nh^N\|u\|_{H_h^{-N}}
$$
\end{lemma}

\subsection{Propagation estimates}
We next recall some propagation estimates for scattering pseudodifferential operators. Since we will work with operators that are fiber classically elliptic, i.e. $\partial(\scatPhase)_\xi\subset \Ells(P)$, we do not need the full scattering calculus here, and will work with operators that are {fiber compactly microlocalized}. In particular, we say that $A\in \Psi^{m,n}$ is \emph{fiber compactly microlocalized}  and write $A\in \Psi^{\comp, n}$ if there is $C>0$ such that
$$
\WFs(A)\cap \{|\xi|>C\}=\emptyset.
$$
For fiber compactly microlocalized operators, all propagation estimates from the standard calculus (see e.g.~\cite[Appendix E.4]{DyZw:19}) follow using the same proofs but interchanging the roles of $x$ and $\xi$.

Throughout, we let $P\in \Psi^{m,n}$ self-adoint with $\sigma_{m,n}(P)=p$, and write $$\varphi_t:=\exp(t\langle \xi\rangle^{1-m}\langle x\rangle^{1-n}H_p):\scatPhase\to \scatPhase$$ 
for the rescaled Hamiltonian flow.  The following lemma follows as in~\cite[Theorem E.47]{DyZw:19}

\begin{lemma}
\label{l:basicProp}
Let $P\in \Psi^{m,n}$ self-adjiont and suppose that $A,B,B_1\in \Psi^{\comp,0}$.
Furthermore, assume that for all $\rho\in \WFs(A)$, there is $T\geq 0$ such that 
$$
\varphi_{-T}(\rho)\in \Ells(B),\qquad \bigcup_{t\in[-T,0]}\varphi_t(\rho)\subset \Ells(B_1).
$$
Then for all $N$ there is $C>0$ such that for $\e\geq 0$, $u\in \mc{S}'$ with $Bu\in H^{s,k}_h$, $B_1(P-i\e\langle x\rangle^{n})u\in H_h^{s,k-n+1}$
$$
\|Au\|_{H_h^{s,k}}\leq C\|Bu\|_{H_h^{s,k}}+Ch^{-1}\|B_1(P-i\e \langle x\rangle^n)u\|_{H_h^{s,k-n+1}}+C_Nh^N\|u\|_{H_h^{-N,-N}}.
$$
\end{lemma}

We will also need the radial point estimates in the setting of fiber compactly microlocalized operators. The following two lemmas are a combination of~\cite[Theorem E.52, E.54]{DyZw:19}  together with the arguments in~\cite[Section 3.1]{DyZw:19b}
\begin{lemma}
\label{l:source}
Let $P\in \Psi^{m,n}$ self adjoint with $n>0$ and let 
$$
L\Subset \{ \langle x\rangle^{-n}p=0\}\cap \partial(\scatPhase)_x
$$
be a radial source for $p$.
Let $k'>\frac{n-1}{2}$, fix $B_1\in \Psi^{\comp,0}$ such that $L\subset \Ells(B_1)$. Then there is $A\in \Psi^{\comp,0}(M)$ such that $L\subset \Ells(A)$ and for all $N$, $k>k'$, $\e\geq0$, and $u\in \mc{S}'$ such that $B_1u\in H_h^{s,k'}$ and $B_1({{P}}-i\e\langle x\rangle^{n})u\in H_h^{s,k-n+1}$,
$$
\|Au\|_{H_h^{s,k}}\leq Ch^{-1}\|B_1(P-i\e\langle x\rangle^n)u\|_{H_h^{s,k-n+1}}+C_Nh^N\|u\|_{H_h^{-N,-N}}.
$$
\end{lemma}

\begin{lemma}
\label{l:sink}
Let $P\in \Psi^{m,n}$ as above with $n>0$, let 
$$
L\Subset \{ \langle x\rangle^{-n}p=0\}\cap \partial(\scatPhase)_x
$$
be a radial sink for $p$
Let $k<\frac{n-1}{2}$, fix $B_1\in \Psi^{\comp,0}$ such that $L\subset \Ells(B_1)$. Then there are $A,B\in \Psi^{\comp,0}(M)$ such that $L\subset \Ells(A)$, $\WFs(B)\subset \Ells(B_1)\setminus L$, and for all $N$, $\e\geq 0$, and $u\in \mc{S}'$ such that $Bu\in H_h^{s,k}$ and $B_1({{P}}-i\e\langle x\rangle^n)u\in H_h^{s,k-n+1}$,
$$
\|Au\|_{H_h^{s,k}}\leq C\|Bu\|_{H_h^{s,k}}+ Ch^{-1}\|B_1(P-i\e\langle x\rangle^n)u\|_{H_h^{s,k-n+1}}+C_Nh^N\|u\|_{H_h^{-N,-N}}.
$$
\end{lemma}

\section{Almost periodic potentials}
\label{s:almostPeriodic}
\subsection{Assumptions on the potential}
We now introduce the objects necessary for our assumptions on the perturbation $W$. We say that $\Theta\subset \mathbb{R}$ is a \emph{frequency set} if
 $\Theta$ is countable, $\Theta=-\Theta$ and $0\in \Theta$.  We write $\Theta^k:=\Theta\times \underset{k-2}{\cdots}\times\Theta$ and $\Theta_k:=\Theta+ \underset{k-2}{\cdots}+\Theta$ and

For a frequency set $\Theta$, and a seminorm $\semi$, on $S^{m,n}$, we will need a family of maps $s_{k,\semi}:\Theta^k\times (S^{m,n})^{\Theta}\to [0,\infty)$. We denote an element $(w_\theta)_{\theta\in\Theta}\in (S^{m,n})^\Theta$ by $\mc{W}$. 
Fix a seminorm, $\semi$ and define
$$
s_{0,\semi}(\mc{W})=1,\qquad s_{1,K}(\theta,\mc{W})=\begin{cases}\frac{\|w_\theta\|_{\semi}}{|\theta|}& \theta\neq 0\\0&\theta=0\end{cases}.
$$
Next, for $\alpha\in \mathbb{N}^j$ with $|\alpha|=k$, define $\beta_i(\alpha)=\sum_{\ell=1}^{i-1}\alpha_\ell$. Then, for $\theta\in \Theta^k$, we write $\theta_{\alpha,i}:=(\theta_{\beta_i(\alpha)+1},\dots \theta_{\beta_{i+1}(\alpha)})\in \Theta^{\alpha_i}$. We can now define 
\begin{gather*}
s_{\alpha,\semi}(\theta,\mc{W}):=\prod_{i=1}^j s_{\alpha_i,\semi}(\theta_{\alpha,i},\mc{W}),\\
s_{k,\semi}(\theta,,\mc{W})= \begin{cases}\frac{1}{|\sum_{i=1}^k\theta_i|}\sum_{p\in Sym(k)}\sum_{|\alpha|=k,\alpha_i\leq k/2}s_{\alpha,\semi}(p(\theta))&\sum_{i=1}^k\theta_i\neq 0\\0&\sum_i\theta_i=0.
\end{cases}
\end{gather*}
where $Sym(k)$ denotes the symmetric group on $k$ elements.

The following two lemmas on the behavior of $s_{k,\semi}$ will be useful below. Their proofs are elementary and we postpone them to Appendix~\ref{a:properties}.

\begin{lemma}
\label{l:basicBound}
There are $C_k, N_k>0$ such that for $\theta\in \Theta^k$,
\begin{equation}
\label{e:basicBound}
|s_{k,\semi}(\theta,\mc{W})|\leq C_k\frac{\prod_{i=1}^k\|w_{\theta_i}\|_{\semi}}{\inf \{|\omega|^{N_k}\mid\omega\in\{\theta_1,0\}+\dots+\{\theta_k,0\}\setminus 0\}}
\end{equation}
\end{lemma}

\begin{lemma}
\label{l:induct}
Suppose that $\tilde{\mc{W}}\in (S^{m,n})^{\Theta_n}$ with
$
(\tilde{\mc{W}})_{\theta_1+\dots+\theta_n}=\tilde{w}_{\theta_1\dots\theta_n}
$
such that for all $\semi$ there is $\semi'$ satisfying
$$
\|\tilde{w}_{\theta_1\dots\theta_n}\|_{\semi}\leq \frac{\prod_{i=1}^n \|w_{\theta_i}\|_{\semi'}}{|\theta_i|}.
$$
Then for all $\semi$, there is $\semi'$ such that 
$$
s_{k,\semi}(\theta_1+\dots+\theta_n,\tilde{\mc{W}})\leq s_{nk,\semi'}((\theta_1,\dots,\theta_n),\mc{W}).
$$
\end{lemma}

We say that $W\in \Psi^1$ is \emph{admissible} if
\begin{equation}
\label{e:wForm}
W= \sum_{\theta\in \Theta} e^{i\theta x}w_\theta(x,hD)
\end{equation}
where $w_\theta\in S^{1,0}$ and
for all $0\leq k$, $\semi$, and $N>0$ we have 
\begin{equation}
\label{e:estimates}
\sum_{\theta\in \Theta^k}s_{k,\semi}(\theta,\mc{W})\leq C_{k,\semi},\qquad \|w_\theta\|_{\semi}<C_{N,\semi}\langle \theta\rangle^{-N},
\end{equation}
where $\mc{W}=(w_\theta)_{\theta\in\Theta}$.

\begin{remark}
When $W$ is smooth and periodic i.e.  $\Theta= r\mathbb{Z}$, and $\|w_\theta\|_{\semi}\leq C_{N,\semi}\langle \theta\rangle^{-N}$, then $W$ is admissible.\end{remark}
\begin{remark}
\label{r:1}
If $W$ is an approximately almost periodic function  of the form
$$
W=\sum_{{\bf{n}}\in \mathbb{Z}^d}e^{i{\bf{n}}\cdot \omega x}w_{\bf{n}}(x,hD)
$$
with $\|w_{\bf{n}}\|_{\semi}\leq C_{N,\semi}\langle {\bf{n}}\rangle^{-N}$ and if $\omega=(\omega_1,\omega_2,\dots \omega_d)$ satisfies the diophantine condition~\eqref{e:diophantine},
then $W$ is admissible. To see this, without loss of generality, we assume that $\omega\in B(0,1)$. Then if $\theta\in \Theta$, $\theta= {\bf{n}}\cdot \omega$ for some  ${\bf{n}}\in\mathbb{Z}^d$.  In particular, if
$$
\theta_{{\bf{n}}_1},\dots, \theta_{{\bf{n}}_k}\in \Theta,\qquad \sum_{i=1}^k \theta_{{\bf{n}}_i}= {\sum_i\bf{n}_i}\cdot \omega,
$$
and hence, if $\sum_i\theta_{{\bf{n}_i}}\neq 0$, then $|\sum_{i=1}^k\theta_{{\bf{n}}_i}|\geq C|\sum_i{\bf{n}}_i|^{-\mu}.$

Using this, observe that by~\eqref{e:basicBound} there are $C_k$, $N_k$ such that 
$$
s_{k,\semi}(\theta_1,\dots,\theta_k)\leq C_k(\sum_i|{\bf{n}}_i|)^{\mu N_k}\prod_{i=1}^k C_N\langle {{\bf{n}}_i}\rangle^{-N}\leq C_k\prod_{i=1}^{k}C_N\langle {{\bf{n}}_i}\rangle^{-N+N_k\mu}
$$

We thus obtain the desired estimate by taking $N>N_{k}\mu+d$ and summing over ${\bf{n}}_i$, $i=1,\dots k$.

\end{remark}

\begin{remark}
\label{r:2}
Next, we verify that certain approximately limit periodic functions are admissible. Suppose that $\{m_n\}_{n=1}^\infty\subset \mathbb{Z}$ contains $0$ and satisfies $\{m_n\}_{n=1}^\infty=\{-m_n\}_{n=1}^\infty$. Suppose 
$$
W=\sum_n e^{im_n x/n}w_n(x,hD)
$$
and $\|w_n\|_K\leq C_{N,K}\langle \max(n, |m_n|/n)\rangle^{-N}$, then $w_n$ satisfies our conditions with $\mu_M\equiv 0$. Indeed, in this case, $\Theta= \{m_n/n\}_{n=1}^\infty$. Now, note that for $\theta_i\in \Theta$, $\theta_i=m_{n_i}/n_{i}$ 
$$
\sum_{i=1}^k\theta_i\neq 0\qquad\qquad \Rightarrow\qquad\qquad \Big|\sum_i\theta_i\Big|\geq \frac{1}{n_1n_2\cdots n_k}
$$
Using this, observe that by~\eqref{e:basicBound} there are $C_k$, $N_k$ such that 
$$
s_{k,\semi}(\theta_1,\dots,\theta_k)\leq C_k(n_1n_2\cdots n_k)^{N_k}{\|w_{\theta_1}\|_{\semi}\cdots\|w_{\theta_k}\|_{\semi}}
$$
In particular, for $N>N_k$
\begin{align*}
s_{k,\semi}(\theta_1,\dots,\theta_k)&\leq C_k\prod_{i=1}^k C_N^kn_i^{N_k} \langle \max(n_i, m_{n_i}/n_i)\rangle^{-N}\leq C_{N,k} \prod_{i=1}^k \langle n_i\rangle^{N_k-N}.
\end{align*}
We thus obtain the desired estimate by taking $N>N_{k}+1$ and summing over $n_i$, $i=1,\dots k$.
\end{remark}

\begin{theorem}
\label{t:main}
Suppose that $W(x,hD)\in \Psi^1$ is self-adjoint and admissible (i.e.~\eqref{e:wForm} and ~\eqref{e:estimates} hold). Let $0<\delta<1$,
$$
P:=-h^2\Delta+hW(x,hD).
$$
Then there are $a_j\in C_c^\infty(\mathbb{R}^3)$ such that for all $R>0$ there is $T>0$ satisfying for all $E\in[1-\delta,1+\delta]$, $\hat{\rho}\in C_c^\infty(\mathbb{R};[0,1])$ with $\hat\rho\equiv 1 $ on $[-T,T]$, and all $x,y\in B(0,R)$ the spectral projector $1_{(-\infty,E]}(P)$ satisfies
$$
1_{(-\infty,E]}(P)(x,y)=h^{-2}\int_{-\infty}^E\int \hat{\rho}(t) e^{it(\mu-|\xi|^2)+(x-y)\xi)/h}a(x,y,\xi;h)d\xi dtd\mu +O(h^\infty)_{C^\infty},
$$
where $a\sim \sum_j h^ja_j$. 
\end{theorem}

After putting $h=\lambda^{-1}$, $W(x,hD)=h(W_1(x)hD_x+hD_xW_1(x))+h^2W_0(x)$, an application of the method of stationary phase, the analysis in Remarks~\ref{r:1} and~\ref{r:2}, and an application Theorem~\ref{t:main} proves Theorems~\ref{t:almostPeriodic} and~\ref{t:limitPeriodic}.
(See~\cite{Iv:18} for a related problem.)
\section{Gauge transforms}

Before gauge transforming our operator, we need the following symbolic lemma which allows us to solve away errors.
\begin{lemma}
\label{l:integrate}
Suppose that $a\in S^{k,0}$. Then, there is $b\in S^{k,0}$ such that $ (D_x+\theta)b-a=r\in S^{k,-\infty}$ and
$$
\|b\|^{k,0}_{\beta,\alpha} \leq C_{\alpha\beta k}|\theta|^{-1}\|a\|^{k,0}_{\beta,\alpha+2},\qquad \|r\|^{k,-N}_{\beta,\alpha}\leq C_{\alpha\beta N }|\theta|^{-1}\|a\|^{k,0}_{\beta,\alpha+N+2}f
$$
\end{lemma}
\begin{proof}

\noindent {\bf{Case 1:}} $|\theta|\geq 1$. Let $\chi \in C_c^\infty(\mathbb{R})$ with $\chi \equiv 1$ on $[-1/3,1/3]$ and $\supp \chi \subset (-1,1)$. Then define
$$
b(x,\xi):=\frac{1}{2\pi}\int e^{i(x-y)\eta}\frac{1-\chi(\theta+\eta)}{\eta+\theta}a(y,\xi)dy d\eta.
$$
where the integral in $y$ interpreted as the Fourier transform. Then, $(D_y+\theta)b-a=r$ where
$$
r(x,\xi):=-\frac{1}{2\pi}\int e^{i(x-y)\eta}\chi(\theta+\eta)a(y,\xi)dy d\eta=-\frac{1}{2\pi}\int e^{i(x-y)\eta}\chi(\theta+\eta)|\eta|^{-N}D_y^Na(y,\xi)dy d\eta
$$
Then, since $\phi_\theta:=\chi(\theta)|\eta-\theta|^{-N}$ is smooth and compactly supported with seminorms bounded uniformly in $|\theta|\geq 1$,
\begin{align*}
| D_x^\alpha D_\xi^\beta r(x,\xi)|&=\Big|-\frac{1}{2\pi}\int e^{i\theta(x-y)}\hat{\phi}_{\theta}(y-x) D_y^{\alpha+N} D_\xi^\beta a(y,\xi)dy \Big|\\
&\leq C_{N,M}\int |\theta|^{-1}\langle x-y\rangle^{-M}\langle y\rangle^{-\alpha-N}\langle \xi\rangle^{k-\beta} \|a\|^{k,0}_{\beta,\alpha+N+1}dy\\
& \leq C|\theta|^{-1}\langle x\rangle^{-\alpha-N}\langle\xi\rangle^{k-\beta}\|a\|^{k,0}_{\beta,\alpha+N+1},
\end{align*}
and
\begin{align*}
|D_x^\alpha D_\xi^\beta b(x,\xi)|&=\Big|\frac{1}{2\pi}\int e^{i(x-y)\eta}\Big(\frac{1-(x-y)D_\eta}{1+|x-y|^2}\Big)^N\frac{1-\chi(\theta+\eta)}{\eta+\theta}\Big(\frac{1 +\eta D_y}{1+|\eta|^2}\Big)^2D_y^\alpha D_\xi^\beta a(y,\xi)dy d\eta\Big|\\
&\leq C_N \Big|\int \langle x-y\rangle^{-N}\langle \eta\rangle^{-2}\langle \eta+\theta\rangle^{-1}\langle \xi\rangle^{k-\beta}\langle y\rangle^{-\alpha}\|a\|_{\beta,\alpha+2}^{k,0}dy d\eta\Big|\\
&\leq C|\theta|^{-1}\langle \xi\rangle^{k-\beta}\langle x\rangle^{-\alpha}\|a\|_{\beta,\alpha+2}^{k,0}.
\end{align*}

\noindent {\bf{Case 2:}} $|\theta|\leq 1$. Define
$
L:S^{k,\ell}\to C^\infty (\mathbb{R}^2)
$
by 
$$
L\tilde{a}:=i\int_{0}^x e^{i\theta (s-x)} \tilde{a}(s,\xi)ds .
$$
Then, 
$
(D_x+\theta)L\tilde{a}=\tilde{a}.
$

Moreover, if $\tilde{a}$ vanishes at $x=0$, then
\begin{align*}
 L\tilde{a}&=i\int_{0}^x [ \frac{D_s}{\theta}e^{i\theta (s-x)}]\tilde{a}(s,\xi)ds=-i\theta^{-1}\int_{0}^x e^{i\theta (s-x)}D_s \tilde{a}(s,\xi)ds +\frac{1}{\theta} \tilde{a}(x,\xi)\\
&=-\theta^{-1}LD_x\tilde{a}+ \theta^{-1}\tilde{a}
\end{align*}
In particular, 
$$
D_xL\tilde{a}=\tilde{a}-\theta L\tilde{a}= LD_x \tilde{a}
$$

Now, suppose that $\tilde{A}\in S^{k,0}$ and $\tilde{a}$ vanishes to infinite order at $x=0$.  Then, for $xr\geq 0$ with $|x|\leq |r|$ 
\begin{align*}
|L\tilde{a}(x,\xi)|&\leq |r|\|\tilde{a}\|^{k,0}_{0,0}\langle \xi\rangle^k
\end{align*}
For $|x|\geq |r|$, 
\begin{align*}
\Big|L\tilde{a}(x,\xi)\Big|&\leq \Big|\int_0^r e^{i\theta s} \tilde{a}(s,\xi)ds\Big|+|\theta|^{-1}\Big(\Big|\int_r^x e^{i\theta s}D_s\tilde{a}(s,\xi)ds\Big|+|\tilde{a}(x,\xi)|+|\tilde{a}(r,\xi)|\Big)\\
&\leq (|r|+2|\theta|^{-1}) \|\tilde{a}\|^{k,0}_{0,0}\langle \xi\rangle^k+|\theta|^{-2}\Big(\Big|\int_r^x e^{i\theta s}D^2_s\tilde{a}(s,\xi)ds\Big|+|D_x\tilde{a}(x,\xi)|+|D_x\tilde{a}(r,\xi)|\Big)\\
&\leq \langle \xi\rangle^k\big( (|r|+2|\theta|^{-1}) \|\tilde{a}\|^{k,0}_{0,0} +|\theta|^{-2}\langle \xi\rangle^k(C\|D_x^2\tilde{a}\|^{k,-2}_{0,0}\langle r\rangle^{-1}+2\|D_x\tilde{a}\|^{k,-1}_{0,0}\langle r\rangle^{-1})\big)
\end{align*}
Optimizing in $r$, we obtain $|r|=|\theta|^{-1}$ and in particular,
$$
\|L\tilde{a}\|_{0,0}^{k,0}\leq C|\theta|^{-1}\|\tilde{a}\|_{0,2}^{k,0}
$$
Therefore, since $D_\xi$ commutes with $L$, if $b\in S^{k,0}$ vanishes to infinite order at $x=0$, we have
$$
\|L \tilde{a}\|^{k,0}_{\beta,0}\leq C|\theta|^{-1}\|\tilde{a}\|^{k,0}_{\beta,2}
$$

Now, consider
$$
D_x L\tilde{a}= \theta^{-1}(-LD_x^2\tilde{a}+D_x\tilde{a}) 
$$
and define 
$$
\tilde{a}_\pm(\xi):=i\int _0^{\pm \infty} e^{i\theta s}D_s^2 \tilde{a}(s)ds.
$$
Arguing as above, we can see that
$$
|\partial_\xi^\beta \tilde{a}_{\pm}(\xi)|\leq C|\theta|\|\tilde{a}\|_{\beta,2}^{k,0}\langle \xi\rangle^{k-\beta}.
$$

Fix $c_{\pm}(x)\in C_c^\infty $ such that $\int c_{\pm}dx=1$, $\supp c_\pm \subset \{\pm x>0\}$. Then, 
$$
\int e^{i\theta s} D_s^2c_{\pm}(s)ds\geq c|\theta|^2,
$$
and putting $\tilde{a}_{mod}(x,\xi)=\tilde{a}(x,\xi)-\frac{c_+(x)\tilde{a}_+(\xi)}{\int e^{i\theta s} D_s^2c_{+}(s)ds}-\frac{c_-(x)\tilde{a}_-(\xi)}{\int e^{i\theta s} D_s^2c_{-}(s)ds}$,
we have 
$$
\int_0^\infty e^{i\theta s} D_s^2D_\xi^\beta \tilde{a}_{mod}(x,\xi)ds=\int_0^{-\infty}e^{i\theta s} D_s^2D_\xi^\beta\tilde{a}_{mod}(s,\xi)ds=0.
$$
Moreover, since $\tilde{a}_{mod}$ vanishes to infinite order at $0$, we can integrate by parts to see that
$$
\int_0^\infty e^{i\theta s}D_s^kD_\xi^\beta \tilde{a}_{mod}ds=\int_0^{-\infty}e^{i\theta s} D_s^kD_\xi^\beta\tilde{a}_{mod}ds=0,\qquad k\geq 2.
$$

Finally, note that for $\alpha\geq 1$,
$$
D_x^\alpha L\tilde{a}_{mod}=\theta^{-1}(-LD_x^{\alpha+1}\tilde{a}_{mod}+D_x^\alpha\tilde{a}_{mod})
$$
and we have 
$$
\Big|\int_0^x e^{is\theta}D_s^{\alpha+1}D_\xi^\beta\tilde{a}_{mod}ds\Big|=\Big|\int_x^{\sgn x \infty }e^{is\theta}D_s^{\alpha+1}D_\xi^\beta\tilde{a}_{mod}ds\Big|\leq C_N\|\tilde{a}\|_{\beta,\alpha+1}^{k.0}\langle x\rangle^{-\alpha}\langle \xi\rangle^{k-\beta}.
$$

To complete the proof we let $\chi \in C_c^\infty(\mathbb{R})$ with $\chi \equiv 1$ near $0$ and put $\tilde{a}=(1-\chi(x))a(x,\xi)$, $b=L\tilde{a}_{mod}.$

\end{proof}

Next, we need a lemma which controls scattering symbols after conjugation by $e^{i\theta x}$.

\begin{lemma}
\label{l:conj}
Suppose that $B\in \Psi^{n,m}$ and $\theta \in \mathbb{R}$, $|\theta|\leq Ch^{-1}$. Then, there is $B_\theta\in \Psi^{n,m}$ such that 
$$
e^{i\theta x}Be^{-i\theta x}=B_{\theta}.
$$
and $\WFs(B_\theta)=\WFs(B).$ Moreover, if $B=b(x,hD)$, then $B_\theta=b_\theta(x,hD)$ where
$$
b_\theta(x,\xi) =b(x,\xi-h\theta)\sim \sum_{j=0}^\infty \frac{h^j(-1)^j}{j!}\langle \theta,\partial_\xi\rangle^j b.
$$
In particular, 
$$
\|b-b_{\theta}\|^{n-1,m}_{\alpha,\beta}\leq \|b\|^{n,m}_{\alpha,\beta+1}h|\theta|\langle h|\theta|\rangle^{n-|\beta|-1}.
$$
\end{lemma}
\begin{proof}
Write 
$
B=b(x,hD)+O(h^\infty)_{\Psi^{-\infty,-\infty}}.
$
Then, 
$$
e^{i\theta x}b(x,hD)e^{-i\theta x}=b_\theta (x,hD),\qquad b_\theta(x,\xi)=b(x,\xi-h\theta).
$$
Now, 
\begin{align*}
|\partial_x^\alpha \partial_\xi ^\beta b_\theta (x,\xi) |&=\Big|\partial_y^\alpha \partial_\eta^\beta b(y,\eta)_{y=x,\eta=\xi+h\theta}\Big|\leq C_{\alpha \beta}\langle x\rangle^{m-|\alpha|}\langle \xi-h\theta\rangle^{n-|\beta|}\\
&\leq C_{\alpha \beta}\langle h\theta\rangle^{n-|\beta|}\langle x\rangle^{m-|\alpha|}\langle \xi\rangle^{n-|\beta|}\leq \tilde{C}_{\alpha\beta}\langle x\rangle^{m-|\alpha|}\langle \xi\rangle^{n-|\beta|}
\end{align*}
and the first part of the lemma follows from Taylor's theorem.

Note also that 
\begin{align*}
\partial_{x}^\alpha\partial_\xi^\beta(b(x,\xi)-b_\theta(x,\xi)) &=h\int_{0}^1 -\langle \partial_x^\alpha\partial^{\beta+1}_{\xi}b(x,\xi-th\theta),\theta\rangle dt\\
&\leq \|b\|^{n,m}_{\alpha,\beta+1}h|\theta|\langle \xi\rangle^{n-|\beta|-1}\langle x\rangle^{m-|\alpha|}\langle h\theta\rangle^{n-|\beta|-1}
\end{align*}
\end{proof}

\begin{lemma}
\label{l:commute}
Suppose that $\theta_1,\theta_2\in B(0,Dh^{-1})$ and that $a \in S^{m_1,n_1}$ and $b\in S^{m_2,n_2}$. Then, 
$$
h^{-1}[e^{i\theta_1 x}Op_h(a),e^{i\theta_2x}Op_h(b)]=e^{i(\theta_1+\theta_2)x}(h^{-1}[Op_h(a),Op_h(b)]+|\theta|c_2(x,hD))
$$
where the map $L:S^{m_1,n_1}\times S^{m_2,n_2}\to S^{m_1+m_2-1,n_1+n_2}$, $(a,b)\mapsto c_2$ is bounded uniformly in $h$ with bound depending only on the constant $D$.
\end{lemma}
\begin{proof}
Note that 
\begin{align*}
[e^{i\theta_1 x}Op_h(a),e^{i\theta_2x}Op_h(b)]&=e^{i(\theta_1+\theta_2)x}(Op_h(a_{-\theta_2})Op_h(b)-Op_h(b_{-\theta_1})Op_h(a))\\
&=e^{i(\theta_1+\theta_2)x}([Op_h(a),Op_h(b)]+(Op_h(a_{-\theta_2}-a))Op_h(b)\\
&\qquad-(Op_h(b_{-\theta_1}-b))Op_h(a))).
\end{align*}
We now apply Lemma~\ref{l:conj} to finish the proof.
\end{proof}

Using Lemma~\ref{l:commute}, we can see that if $\Theta_1,\Theta_2\subset B(0,Dh^{-1})$
\begin{equation}
\label{e:conjugateForm}
G=\sum_{\theta\in \Theta_1} e^{i\theta x}g_\theta(x,hD),\quad g_\theta \in S^{m_1,n_1}\qquad B=\sum_{\theta\in \Theta_2}e^{i\theta x}b_\theta(x,hD),\quad b_\theta\in S^{m_2,n_2}
\end{equation}
 then,
 $$
 h^{-1}[G,B]= \sum_{\theta_i \in \Theta_1
 ,\theta_j\in \Theta_2}e^{i(\theta_1+\theta_2)x}\tilde{g}_{\theta_1,\theta_2}(x,hD)
 $$
where, for all $m_i,n_i$, $i=1,2$ and $\alpha,\beta\in \mathbb{N}$, there are $K,C>0$ such that 
$$
\|\tilde{g}_{\theta_1,\theta_2}\|^{m_1+m_2-1,n_1+n_2}_{\alpha\beta}\leq C(1+\max(|\theta_1|,|\theta_2|)\|g_{\theta_1}\|_{\beta+K,\alpha+K}^{m_1,n_1}\|b_{\theta_2}\|_{\beta+K,\alpha+K}^{m_2,n_2}
$$

Thus, applying Lemma~\ref{l:conjugate}, we have the following lemma:
\begin{lemma}
\label{l:conjugateFull}
Let  $G\in \Psi^{-\infty}$ self-adjoint and $B$ are as in~\eqref{e:conjugateForm} with $m_1=m_2=-
\infty$ and $n_1=n_2=0$. Then,
$$
e^{iG}Be^{-iG}=B+\sum_{j=1}^{k-1}\sum_{\substack{{\bf\Phi}\in \Theta_1^j\\\theta\in \Theta_2}}h^je^{i(\sum_{i=1}^{j}{\bf{\Phi}_i}+\theta)x} \tilde{g}_{{\bf\Phi},\theta} +O(h^{k})_{H_h^{-N}\to H_h^{N}}
$$
where for any $\sum_{i=0}^j N_i=N$, $\alpha,\beta$ there are $K$ and $C_{N\alpha\beta j}$ such that
$$
\|\tilde{g}_{{\bf\Phi},\theta}\|^{-N,0}_{\beta,\alpha}\leq C_{j\alpha\beta}(1+|\theta|)\|b_{\theta}\|_{\beta+K,\alpha+K}^{-N_0,0}\prod_{i=1}^j(1+|{\bf\Phi}_i|)\|g_{{\bf{\Phi}}_i}\|_{\beta+K,\alpha+K}^{-N_i,0}.
$$
\end{lemma}

\subsection{The gauge transform}

We are now in a position to prove the inductive lemma used for gauge transformation.
\begin{lemma}
\label{l:gauge}
Suppose that $0<a<b$ and 
$$
\WF\big(\tilde{P}-(P_0+hQ_k+h^{1+k}W_k+h^NR_k\big)\cap \{|\xi|\in [a,b]\}=\emptyset
$$
where $Q_k\in \Psi^{-\infty,0}$, $R_k\in \Psi^{-\infty}$, 
$$
W_k=\sum_{\theta \in \Theta \setminus 0} e^{i\theta x}w_{\theta,k}(x,hD)
$$
with $\{w_{\theta,k}\}_{\theta\in \Theta}$ satisfying~\eqref{e:estimates} and $W_k,Q_k$ self adjoint.
Then there is $G\in h^{-\delta+k(1-\delta)}S^{-\infty,0}$ self adjoint such that 
$$
\WF\big(\tilde{P}_G-(P_0+hQ_{k+1}+h^{1+k+1}W_{k+1} +h^{N}R_{k+1})\big)\cap \{|\xi|\in [a,b]\}=\emptyset
$$
where $R_{k+1}\in \Psi^{-\infty}$, 
$$
Q_{k+1}=Q_k+h^{k+1}\tilde{Q}_k\in \Psi^{-\infty,0}
$$
with $\tilde{Q}_k$ self adjoint and $W_{k+1}$ is self adjoint with
$$
 W_{k+1}=\sum_{\theta\in \Theta_{\lceil \frac{N}{k+1}-1\rceil}\setminus 0}e^{i\theta x}w_{\theta,k+1}(x,hD)
$$ 
satisfies~\eqref{e:estimates} with  $\Theta,$ replaced by
$
\Theta_{\lceil \tfrac{N}{k+1}-1\rceil}.
$
\end{lemma}
\begin{proof}
Let $\chi \in C_c^\infty(0,\infty)$ such that $\chi \equiv 1$ near $[1/2,2]$ and 
$$
\tilde{P}\chi(|hD|)=(P_0+hQ_k+h^kW_k-h^NR_k)\chi(|hD|)+O(h^\infty)_{\Psi^{-\infty}}
$$
We aim to use the fact that $P_0$ dominates $\tilde{P}$ to conjugate away $W_k$.
Therefore, we look for $G$ such that, modulo lower order terms,
$$
ih^{-1-k}[P_0,G]=W_k.
$$
To do this, we solve
$$
2\xi\partial_x g= \sigma_{-\infty}(W_k\chi(|hD|)).
$$
Now, 
$$
W_k\chi(|hD|)=\sum_{\theta\in \Theta\setminus \{0\}}e^{i\theta x} (w_{\theta}\chi(|\xi|))(x,hD)
$$
where $w_{\theta}\in S^{-\infty, 0}$ satisfy~\eqref{e:estimates}. Let $\chi_i\in C_c^\infty(0,\infty)$, $i=1,2$, such that $\chi_1,\chi_2\equiv 1$ near $[a,b]$ and $\supp\chi_2\subset \supp \chi_1\subset \supp \chi$. By, Lemma~\ref{l:integrate}, there is $g_\theta\in S^{-\infty, 0}$ such that 
$$
(D_x+\theta)g_\theta(x,\xi)-iw_{\theta,k}\chi_1(|\xi|)/2\xi\in S^{-\infty,-\infty},\qquad
\|g_\theta\|^{-N,0}_{\beta,\alpha} \leq C_{\alpha\beta N}|\theta|^{-1}\|w_{\theta}\chi_1\|^{-N,0}_{\beta,\alpha+2}.
$$

Modifying lower order terms in $g_{\theta}$ to make $e^{i\theta x} g_\theta+e^{-i\theta x}g_{-\theta}$ self adjoint, we put
$$
G:=h^{k}\sum_{\theta\in \Theta\setminus \{0\}} e^{i\theta x}g_\theta(x,hD).
$$
Then,  $G\in h^{k}S^{-\infty}$, and, letting $\tilde{k}=(k+1)$, by Lemma~\ref{l:conjugateFull}, for any $N_1$
\begin{align*}
\chi_2(|hD|)\tilde{P}_G&=\chi_2(|hD|)(P_0+hQ_k)+\sum_{j=2}^{N_1-1}\sum_{\substack{{\bf\Phi}\in \Theta^j}}e^{i(\sum_{i=1}^{j}{\bf{\Phi}_i})x} h^{j\tilde{k}}\tilde{g}^1_{{\bf\Phi}} (x,hD)\\
&+\sum_{j=1}^{N_1-1}\sum_{{\bf\Phi}\in \Theta^{j}}h^{j\tilde{k}+1}e^{i(\sum_{i=1}^{j}{\bf{\Phi}_i})x} \tilde{g}^2_{{\bf\Phi}}(x,hD)+O(h^{N_1\tilde{k}})_{H_h^{-N}\to H_h^{N}} +O(h^N)_{H_h^{-N}\to H_h^N}
\end{align*}
where for ${\bf\Phi}\in \Theta^n$,
$$
\|\tilde{g}^\ell_{{\bf\Phi}}\|^{-N,0}_{\alpha\beta}\leq C_{j\alpha\beta N}(1+\|Q_k\|_{\alpha+K,\beta+K}^{-N,0})\prod_{i=1}^{n}(1+|{\bf\Phi}_i|)|{\bf{\Phi}}_i|^{-1}\|w_{{\bf{\Phi}}_i}\chi_1\|_{\beta+K,\alpha+K+2}^{-N,0}.
$$

In particular, putting $N_1=\lceil \frac{N}{k+1}\rceil$, and
$$
W_{k+1}=\sum_{j=2}^{N_1-1}\sum_{\substack{{\bf\Phi}\in \Theta^j\\\sum{\bf{\Phi}}_i\neq0}}e^{i(\sum_{i=1}^{j}{\bf{\Phi}_i})x} h^{j\tilde{k}}\tilde{g}^1_{{\bf\Phi}} (x,hD)\\+\sum_{j=1}^{N_1-1}\sum_{\substack{{\bf\Phi}\in \Theta^{j}\\\sum{\bf{\Phi_i}}\neq 0}}h^{j\tilde{k}+1}e^{i(\sum_{i=1}^{j}{\bf{\Phi}_i})x} \tilde{g}^2_{{\bf\Phi}}(x,hD)
$$
and 
$$
Q_{k+1}=Q_k+\sum_{j=2}^{N_1-1}\sum_{\substack{{\bf\Phi}\in \Theta^j\\\sum{\bf{\Phi}}_i=0}}h^{j\tilde{k}}\tilde{g}^1_{{\bf\Phi}} (x,hD)\\+\sum_{j=1}^{N_1-1}\sum_{\substack{{\bf\Phi}\in \Theta^{j}\\\sum{\bf{\Phi_i}}=0}}h^{j\tilde{k}+1} \tilde{g}^2_{{\bf\Phi}}(x,hD)
$$
we have by Lemma~\ref{l:induct} that $W_{k+1}$ satisfies~\eqref{e:estimates} with  $\Theta$ replaced by
$\Theta=\Theta_{\lceil N/(k+1)-1\rceil}.$
\end{proof}

The following is now an immediate corollary of the previous lemma
\begin{corollary}
\label{c:gauge}
Let $P=-h^2\Delta+hW$ where $W$ is admissible and $0<a<b$. Then for all $N$ there is $G\in \Psi^0$ self-adjoint such that 
$$
e^{iG}Pe^{-iG}= -h^2\Delta +hQ+(1-\chi(h^2\Delta-1))h\tilde{W}(1-\chi(h^2\Delta-1))+O(h^N)_{\Psi^{-\infty}}
$$
where $Q\in \Psi^{-\infty,0}$, $\tilde{W}\in \Psi^1$, are self adjoint,  and $\chi \in C_c^\infty $ with $\chi \equiv 1$ on $[a,b]$.
\end{corollary}

\section{Limiting absorption for the gauge transformed operator}
Throughout this section, we work with an operator
\begin{equation}
\label{e:resolveForm}
\begin{gathered}
P=P_0+h(1-\chi(-h^2\Delta-1))W(x,hD)(1-\chi(-h^2\Delta-1))),\\ P_0\in S^{2,0},\,\sigma_{2,0}(P_0)=|\xi|^2,\qquad \sigma_{1,-1}(h^{-1}\Im P_0)=0.
\end{gathered}
\end{equation}
where $\chi \in C_c^\infty(\mathbb{R})$ with $\chi\equiv 1$ in a neighborhood of $[-\delta,\delta]$. and $W\in \Psi^1$. We will show that for $E\in [1-\delta,1+\delta]$, $R_{\pm}(E):=(P-E\mp  i0)^{-1}$ exist as limiting absorption type limits. Moreover, we will show that $R_{\pm}(E)$ satisfy certain outgoing/incoming properties.

Throughout this section, we let $\chi_i\in C_c^\infty(\mathbb{R})$ $i=1,2,3$ with 
\begin{equation}\label{e:chi}
\begin{gathered}
\chi_i\equiv 1\text{ near }[-\delta,\delta],\qquad \supp \chi_i\subset \{\chi_{i-1}\equiv 1\},\, i=2,3, \qquad \supp\chi_{1}\subset\{ \chi\equiv 1\},\\
\psi_i:=(1-\chi_i((-h^2\Delta-1))),\qquad X_i:=\chi_i((-h^2\Delta-1))
\end{gathered}
\end{equation}

\subsection{Elliptic Estimates}
We first obtain estimates in the elliptic region where the perturbation of $P_0$ is supported.
\begin{lemma}
With $\psi_i$ as in~\eqref{e:chi},
\begin{equation}
\label{e:highFrequency}
c\|\psi_2u\|_{H_h^{s+2,k}}\leq \|\psi_3 (P-E\pm i\e)u\|_{H_h^{s,k}}+Ch^N\|u\|_{H_h^{-N,-N}}.
\end{equation}
\end{lemma}
\begin{proof}
Observe that
$$
\psi _i(P-E)= \psi_i (P_0-E)+h(1-\chi(-h^2\Delta-1))W(x,hD)(1-\chi(-h^2\Delta-1)))
$$
since $\psi_i(1-\chi(-h^2\Delta-1)))=(1-\chi(-h^2\Delta-1))$. Note that $\WFs(\psi_2)\subset\Ells(\psi_3 (P_0-E))$, and hence by Lemma~\ref{l:elliptic}, for $\e>0$,
\begin{equation*}
\begin{aligned}
\|\psi_3 (P-E\pm i\e)u\|_{H_h^{s,k}}&\geq \|\psi_3 (P_0-E\pm i\e)u\|_{H_h^{s,k}}-Ch\|(1-\chi)u\|_{H_h^{s+1,k}}\\
&\geq c\|\psi_2 u\|_{H_h^{s+2,k}}-Ch\|(1-\chi)u\|_{H_h^{s+1,k}}-Ch^N\|u\|_{H_h^{-N,-N}}\\
&\geq c\|\psi_2u\|_{H_h^{s+2,k}}^2-Ch^N\|u\|_{H_h^{-N,-N}}^2.
\end{aligned}
\end{equation*}
Here, in the last line we have used that $(1-\chi)=(1-\chi)\psi_2$.
\end{proof}

\subsection{Propagation estimates}
Consider
$
\tilde{P}_{E}:=\langle x\rangle^{1/2}(P_0-E)\langle x\rangle^{1/2}
$
so that $\tilde{P}_{E}\in \Psi^{-\infty,1}$ is self-adioint and 
$$
\sigma_{2,1}(\tilde{P}_{E})=\langle x\rangle(\xi^2-E)=:\tilde{p}.
$$
Note that
$$
H_{\tilde{p}}=2\xi\langle x\rangle\partial_x -(\xi^2-E)x\langle x\rangle^{-1},
$$
and therefore, letting,
\begin{gather*}
L_+=\bigcup_{\pm}L_{+,\pm},\qquad L_{+,\pm}:=\{\xi=\pm \sqrt{E},x=\pm \infty\},\\
 L_-=\bigcup_{\pm}L_{-,\pm},\qquad L_{-,\pm}:=\{\xi=\mp \sqrt{E},x=\pm \infty\},
\end{gather*}
we have that $L_{+,\pm}$ are radial sinks for $\tilde{p}$ and $L_{-,\pm}$ are radial sources (see~\cite[Definition E.50]{DyZw:19}). 

\begin{lemma}
\label{l:radial}
Let $B_+,B_-\in \Psi^{\comp,0}$, 
\begin{equation}
\label{e:wfs}
\begin{gathered}
L_\pm\subset \Ells(B_\pm),\qquad \WFs(B_\pm)\cap L_\mp=\emptyset,\qquad \{p=E\}\subset (\Ells(B_-)\cup \Ells(B_+))
\end{gathered}
\end{equation}
and $B'_\pm\in\Psi^{\comp,0}$ with the same property, and $\WFs(B'_{\pm})\subset \Ells(B_{\pm}).$ Then, for all $k_+<-\frac{1}{2}$ and $k_->k_-'>-\frac{1}{2}$, and $N$ there is $C>0$ and $\delta>0$ such that for $\e\geq 0$, $E\in[1-\delta,1+\delta]$, and $u\in \mc{S}'(\mathbb{R})$ with $B_\pm (P_0-E-i\e)u\in H_{h}^{0,k_{\pm}}$, and $B_- u\in H_h^{0,k_-'}$, 
\begin{multline*}
\|B_+'u\|_{H_h^{0,k_+}}+\|B_-'u\|_{H_h^{0,k_-}}\\\leq Ch^{-1}(\|B_+(P_0-E-i\e)u\|_{H_h^{0,k_{+}+1}}+\|B_-(P_0-E-i\e)u\|_{H_h^{0,k_{-}+1}})+Ch^N\|u\|_{H_h^{-N,-N}}.
\end{multline*}
Similarly, for all $\tilde{k}_+>\tilde{k}_+'>-\frac{1}{2}$ and $\tilde{k}_-<-\frac{1}{2}$, and $N$ there are $C>0$ and $\delta>0$ such that for $\e\geq 0$, $E\in[1-\delta,1+\delta]$, and $u\in \mc{S}'(\mathbb{R})$ with $B_\pm (P_0-E+i\e)u\in H_{h}^{0,\tilde{k}_{\pm}}$, and $B_+u\in H_h^{0,\tilde{k}_+'}$, 
\begin{multline*}
\|B_+'u\|_{H_h^{0,\tilde{k}_+}}+\|B_-'u\|_{H_h^{0,\tilde{k}_-}}\\\leq Ch^{-1}(\|B_+(P_0-E+i\e)u\|_{H_h^{0,\tilde{k}_{+}+1}}+\|B_-(P_0-E+i\e)u\|_{H_h^{0,\tilde{k}_{-}+1}})+Ch^N\|u\|_{H_h^{-N,-N}}.
\end{multline*}

\end{lemma}
\begin{proof}

Let $\tilde{B}_-\in \Psi^{\comp,0}$ such that $L_{-}\subset \Ells(\tilde{B}_-)$, $\WFs(\tilde{B}_-)\subset \Ells(B_-)$. Then, by Lemma~\ref{l:source} there is $A_-\in\Psi^{\comp,0}$ such that $L_{-}\subset \Ells(A_-)$ and for all $\tilde{k}_->\tilde{k}'_->0$, $\e\geq 0$ and $v\in \mc{S}'(\mathbb{R})$ with $\tilde{B}_-v\in H_h^{0,\tilde{k}_-'}$, $\tilde{B}_-(\tilde{P}_E-i\e\langle x\rangle)v\in H_h^{0,\tilde{k}_-}$, 
\begin{equation}
\label{e:source}
\|A_-v\|_{H_h^{0,\tilde{k}_-}}\leq Ch^{-1}\|\tilde{B}_-(\tilde{P}_E-i\e\langle x\rangle) v\|_{H_h^{0,\tilde{k}_-}}+C_Nh^N\|v\|_{H_h^{-N,-N}}.
\end{equation}

Next, let $\tilde{B}_+\in \Psi^{\comp,0}$ such that $L_{+}\subset \Ells(\tilde{B}_+)$, $\WFs(\tilde{B}_+)\subset \Ells(B_+)$. Then, by Lemma~\ref{l:sink} there are $A_+,B\in\Psi^{\comp,0}$ such that $L_{+}\subset \Ells(A_+)$, $\WFs(B)\subset \Ells(\tilde{B}_+)\setminus L_{+}$, and for all $\tilde{k}_+<0$, $\e\geq 0$ and $v\in \mc{S}'(\mathbb{R})$ with $Bv\in H_h^{0,\tilde{k}_+}$, $\tilde{B}_+(\tilde{P}_E-i\e\langle x\rangle)v\in H_h^{0,\tilde{k}_+}$, 
\begin{equation}
\label{e:sink}
\|A_+v\|_{H_h^{0,\tilde{k}_+}}\leq C\|Bv\|_{H_h^{0,\tilde{k}_+}}+Ch^{-1}\|\tilde{B}_+(\tilde{P}_E-i\e\langle x\rangle) v\|_{H_h^{0,\tilde{k}_+}}+C_Nh^N\|v\|_{H_h^{-N,-N}}.
\end{equation}

Finally, let $B_0\in \Psi^{\comp,0}$ with $\WFs(B_0)\subset \Ells(B_+)$, 
$$
\{\tilde{p}=0\}\subset \Ells(B_0)\cup\Ells(B_-').
$$
Then, there is $A_0\in \Psi^{\comp,0}$ such that $\WFs(A_0)\cap (L_+\cup L_-)=\emptyset$ and there is $T>0$ with
\begin{gather}
\WFs(A_0)\subset \bigcup_{0\leq t\leq T}\varphi_t(\Ells(A_-))\cap \Ells(B_0),\label{e:WFcond}\\
 \{ \langle x\rangle^{-1}\tilde{p}=0\}\subset\Ells(A_0)\cup \Ells(A_-)\cup \Ells(A_+).\label{e:elliptic}
\end{gather}
Now, by~\eqref{e:WFcond} and Lemma~\ref{l:basicProp} for all $\e\geq 0$, and $u\in \mc{S}'(\mathbb{R})$ such that $A_-v\in H_h^{0,\tilde{k}_-}$,  $B_0(\tilde{P}_E-i\e\langle x\rangle) v\in H_h^{0,\tilde{k}_-}$,
\begin{equation}
\label{e:propagate}
\|A_0v\|_{H_h^{0,\tilde{k}_-}}\leq C\|A_-u\|_{H_h^{0,\tilde{k}_-}}+Ch^{-1}\|B_0(\tilde{P}_E-i\e\langle x\rangle) v\|_{H_h^{0,\tilde{k}_-}}+C_Nh^N\|v\|_{H_h^{-N,-N}}.
\end{equation}

Next, observe that if $B_i\in \Psi^{\comp,0}$ with $\WFs(B_1)\subset \Ells(B_2)$, then there is $C_{k,s}>0$ such that for all $w\in \mc{S}'(\mathbb{R})$ with $B_2w\in H_h^{0,k+s}$,
$$
\|B_1\langle x\rangle^{s}w\|_{H_h^{0,k}}\leq C\|B_2w\|_{H_h^{0,k+s}},
$$
Combining~\eqref{e:source},~\eqref{e:sink},~\eqref{e:propagate}, and using~\eqref{e:elliptic} and Lemma~\ref{l:elliptic} finishes the proof of the first inequality after putting $v=\langle x\rangle^{-1/2}u$ and letting $\tilde{k}_+=k_++\frac{1}{2}=$, $\tilde{k}_-=k_-+\frac{1}{2}.$
 
 The second inequality follows by replacing $\tilde{P}$ by $-\tilde{P}$.
 \end{proof}
 
 Now, for each $\Gamma\subset \scatPhase$, let $B_\Gamma\in \Psi^{0,0}$ such that $\WFs(B_\Gamma)\subset \overline{\Gamma}$, $ \Ells(B_\Gamma)=\Gamma^o$. Then, for $k_\Gamma\geq k$, $s\in \mathbb{R}$ define the norm,
 $$
 \|u\|_{\mc{X}_{\Gamma}^{s,k_\Gamma,k}}:=\|B_\Gamma u\|_{H_h^{s,k_\Gamma}}+\|u\|_{H_h^{s,k}}.
 $$
 \begin{lemma}
 \label{l:resolve}
For $k_->-\frac{1}{2}$, $k_+<-\frac{1}{2}$, $\Gamma_-,\Gamma_-'\subset\scatPhase$ open with $L_-\subset\Gamma\Subset \Gamma'\Subset \{\chi_3(|\xi|^2-1)\equiv 1\}\setminus L_+$, there is $h_0>0$ such that for all $u\in \mc{X}_{\Gamma,+}^{s,k_-,k_+}$, $\e>0$, and $0<h<h_0$
 $$
 \|u\|_{\mc{X}_{\Gamma_-}^{s,k_-,k_+}}\leq Ch^{-1}\|(P-E-i\e)u\|_{\mc{X}_{\Gamma_-'}^{s-2,k_-+1,k_++1}}.
 $$
 For $k_-<-\frac{1}{2}$, $k_+>-\frac{1}{2}$ and $\Gamma_+,\Gamma_+'\subset\scatPhase$ open with $L_+\subset\Gamma_+\Subset \Gamma_+'\Subset \{\chi_3(|\xi|^2-1)\equiv 1\}\setminus L_-$, there is $h_0>0$ such that for all $u\in \mc{X}_{\Gamma,-}^{s,k_-,k_+}$, $\e>0$, and $0<h<h_0$
 $$
 \|u\|_{\mc{X}_{\Gamma_+}^{s,k_+,k_-}}\leq Ch^{-1}\|(P-E-i\e)u\|_{\mc{X}_{\Gamma_+'}^{s-2,k_++1,k_-+1}}.
 $$
 \end{lemma}
\begin{proof}
Put $f_\e=(P-E-i\e)u$. Let $\Gamma_-\Subset \Gamma_1\Subset \Gamma_2\Subset\Gamma_-'$ and $A_{\Gamma_1}$, $A_{\Gamma_2}\in \Psi^{\comp,0}$ such that 
\begin{gather*}
\Gamma_-\Subset \Ells(A_{\Gamma_1})\subset\WFs(A_{\Gamma_1})\subset \Gamma_1\subset \Ells(A_{\Gamma_2})\subset \WFs(A_{\Gamma_2})\subset \Gamma_2,\\
 \WFs(\Id-A_{\Gamma_i})\cap L_-=\emptyset,\qquad \WFs(\Id-A_{\Gamma_1})\subset \Ells(\Id-A_{\Gamma_2}) .
\end{gather*}
Next, define
\begin{equation}
\label{e:defB}
B_+:=(\Id-A_{\Gamma_1})X_1,\qquad B_-:=A_{\Gamma_2}X_1,\qquad B_+':=(\Id-A_{\Gamma_2})X_2,\qquad B_-':=A_{\Gamma_1}X_2.
\end{equation}
Then,~\eqref{e:wfs} is satisfied and by Lemma~\ref{l:radial} together with the fact that $X_2P=X_2P_0$,
\begin{equation}
\label{e:1}
\begin{aligned}
\|B_-'u\|_{H_h^{0,k_-}}+\|B_+'u\|_{H_h^{0,k_+}}&\leq Ch^{-1}(\|B_+f_\e\|_{H_h^{0,k_++1}}+\|B_-f_\e \|_{H_h^{0,k_-+1}})+Ch^N\|u\|_{H_h^{-N,-N}}\\
&\leq  Ch^{-1}(\|f_\e\|_{H_h^{0,k_++1}}+\|B_-f_\e\|_{H_h^{0,k_-+1}})+Ch^N\|u\|_{H_h^{-N,-N}}
\end{aligned}
\end{equation}
Now, since $\WFs(A_{\Gamma_i})\subset  \Ells(B_{\Gamma_-'})$, we have by Lemma~\ref{l:elliptic}
\begin{equation}
\label{e:2p}
\begin{aligned}
\|B_-f_\e\|_{H_h^{s,k_-+1}}+\|A_{\Gamma_1}f_\e\|_{H_h^{s,k_-+1}}&\leq C\|B_{\Gamma_-'}f_\e\|_{H_h^{s,k_-+1}}+C_Nh^N\|f_\e\|_{H_h^{-N,-N}}\\
&\leq C\|B_{\Gamma_-'}f_\e\|_{H_h^{s,k_-+1}}+C_Nh^N\|u\|_{H_h^{-N,-N}}.
\end{aligned}
\end{equation}

Next, since $\WFs(A_{\Gamma_2})\cap \WFs(\Id-X_3)=\emptyset$, and $\WFs(\Id-X_2)\subset \WFs(\Id-X_3)$,
we have by~\eqref{e:defB},~\eqref{e:1} and~\eqref{e:2p} that
\begin{equation}
\label{e:2}
\begin{aligned}
\|A_{\Gamma_1}u\|_{H_h^{s,k_-}}&\leq C_s\|A_{\Gamma_1}X_2u\|_{H_h^{0,k_-}}+\|A_{\Gamma_1}(\Id-X_2)u\|_{H_h^{s,k_-}}\\
&\leq Ch^{-1}\|B_{\Gamma'_-}f_\e\|_{H_h^{0,k_-+1}}+\|f_\e\|_{H_h^{0,k_++1}} +Ch^N\|u\|_{H_h^{-N,-N}}\\
&\leq Ch^{-1}\|f_\e\|_{\mc{X}_{\Gamma_-'}^{s-2,k_-+1,k_++1}}+Ch^N\|u\|_{H_h^{-N,-N}}.
\end{aligned}
\end{equation}
Now, since $\WFs(X_2)\subset \Ells(X_1)$, and $\{p=E\}\subset \Ells(\Id-A_{\Gamma_2})\cup \Ells(A_{\Gamma_1})$,
$$
\WFs(X_2)\setminus( \Ells(\Id-A_{\Gamma_2})\cup \Ells (A_{\Gamma_1}))\subset \Ells( X_1(P_0-E-i\e)),
$$
with uniform bounds in $\e\geq 0$. Therefore, using~\eqref{e:1},~\eqref{e:2p} together with the the elliptic estimate from Lemma~\ref{l:elliptic}, we have 
\begin{align*}
\|X_2u\|_{H_h^{0,k_+}}&\leq Ch^{-1}\|f_\e \|_{\mc{X}_{\Gamma_-'}^{s-2,k_-+1,k_++1}}+C\|X_1f_\e\|_{H_h^{0,k_+}}+Ch^N\|u\|_{H_h^{-N,-N}}\\
&\leq Ch^{-1}\|f_\e\|_{\mc{X}_{\Gamma_-'}^{s-2,k_-+1,k_++1}}+Ch^N\|u\|_{H_h^{-N,-N}}.
\end{align*}
So, using~\eqref{e:highFrequency},
\begin{equation}
\label{e:3}
\|u\|_{H_h^{s,k_+}}\leq Ch^{-1}\|f_\e\|_{\mc{X}_{\Gamma_-'}^{s-2,k_-+1,k_++1}}+Ch^N\|u\|_{H_h^{-N,-N}}.
\end{equation}
For $h$ small enough, the first part of the lemma follows from~\eqref{e:2} and~\eqref{e:3}, the fact that $\WFs(B_{\Gamma_-})\subset \Ells(A_{\Gamma_1})$, and the elliptic estimate (Lemma~\ref{l:elliptic}).  The second claim follows from an identical argument.
\end{proof}
\subsection{The limiting absorption principle and the outgoing property}
We are now in a position to prove the limiting absorption principle. For this, we define $R(\lambda):=(P-\lambda)^{-1}:H_h^{s,k}\to H_h^{s+2,k}$ for $\Im \lambda\neq 0$. 

\begin{lemma}
\label{l:limabs}
Let $\Gamma_-$ be a neighborhood of $L_-$ satisfying the assumptions of Lemma~\ref{l:resolve}, $k_->\frac{1}{2}$, $k_+<-\frac{1}{2}$, $s\in \mathbb{R}$, and $E\in [1-\delta,1+\delta]$, the strong limit $R(E+i0): H_h^{s,k_-}\to \mc{X}_{\Gamma_-}^{s+2,k_--1,k_+}$ exists and satisfies the bound
$$
\|R(E+i0)f\|_{\mc{X}_{\Gamma_-}^{s+2,k_--1,k_+}}\leq Ch^{-1}\|f\|_{H_h^{s,k_-}}.
$$
Similarly, for $ \Gamma_+$ a neighborhood of $L_+$ satisfying the assumptions of Lemma~\ref{l:resolve}, $k_-<-\frac{1}{2}$, $k_+>\frac{1}{2}$,  $s\in \mathbb{R}$, and $E\in[1-\delta,1+\delta]$, the strong limit $R(E-i0): H_h^{s,k_+}\to \mc{X}_{\Gamma_+}^{s+2,k_+-1,k_-}$ exists and satisfies the bound
$$
\|R(E-i0)f\|_{\mc{X}_{\Gamma_+}^{s+2,k_+-1,k_-}}\leq Ch^{-1}\|f\|_{H_h^{s,k_+}}.
$$
\end{lemma}
\begin{proof}
 We start by showing that for $k_->\frac{1}{2}$, $k_+<-\frac{1}{2}$, $R(E+i\e):=(P-E-i\e)^{-1}:H_h^{s,k_-}\to H_h^{s+2,k_+}$ converges as $\e \to 0^+$.  First, note that for each fixed $\e>0$, $R(E+i\e):H_h^{s,k}\to H_h^{s+2,k}$ is well defined. Let  $\Gamma_-'$ be a neighborhood of $L_-$ with $ \Gamma_-\Subset \Gamma_-'$.

  Suppose there is $f\in H_h^{s,k_-}$ such that $R(E+i\e)f$ is not bounded in $H_h^{s+2,k_+}$. Then, there are $\e_n\to 0^+$ such that, defining $u_n:=R(E+i\e_n)f\in H_h^{s+2,k_-}$, we have $\|u_n\|_{H_h^{s+2,k_+}}\to \infty$. Putting $v_n=u_n/\|u_n\|_{H_h^{s+2,k_+}}$, we have that $\|v_n\|_{H_h^{s+2,k_+}}=1$ is bounded  and $(P-E-i\e_n)v_n=f/\|u_k\|_{H_h^{s,k_+}}\to 0$ in $H_h^{s,k_-}$. 
  
Since $f\in H_h^{s,k_-}$, for all $\e>0$, $R(E+i\e):H_h^{s,k}\to H_h^{s+2,k}$, and $k_--1> k_+$, we have $v_n\in \mc{X}_{\Gamma_-}^{s+2,k_--1,k_+}$ and $f\in \mc{X}_{\Gamma_-'}^{s,k_-,k_++1}$. Therefore, by Lemma~\ref{l:resolve} all $n$, 
\begin{align*}
\|v_n\|_{H_h^{s+2,k_+}}&\leq C\|v_n\|_{\mc{X}_{\Gamma_-}^{s+2,k_--1,k_+}}\\
&\leq Ch^{-1}\|f\|_{{\mc{X}_{\Gamma_-'}^{s+2,k_-,k_++1}}}/\|u_k\|_{H_h^{s+2,k_+}}\leq Ch^{-1}\|f\|_{H_h^{s,k_-}}/\|u_k\|_{H_h^{s+2,k_+}}\to 0
\end{align*}
which contradicts the fact that $\|v_n\|_{H_h^{s+2,k_+}}=1$. In particular, $u=R(E+i\e)f$ is uniformly bounded in $H_h^{s+2,k_+}$, and, arguing as above
$$
\|u\|_{\mc{X}_{\Gamma_-}^{s+2,k_--1,k_+}}\leq Ch^{-1}\|f\|_{H_h^{s,k_-}}.
$$

Now, we show that $R(E+i\e)f$ converges as $\e\to 0^+$. To see this, first take any sequence $\e_n\to 0^+$. Then, $R(E+i\e_n)f$ is bounded in $X_{\Gamma_-}^{s+2,k_--1,k_+}$ and hence, for any $s'<s$, $\frac{1}{2}<k'<k_-$, and $k_+'<k_+$, we may extract a subsequence and assume that $u_n=R(E+i\e_n)f\to u$ in $X_{\Gamma_-}^{s'+2,k'_--1,k'_+}$, $(P-E)u=f$, and $u_n\rightharpoonup u$ in $X_{\Gamma_-}^{s+2,k_--1,k_+}$.

Suppose that there is another sequence which converges to $u'\in X_{\Gamma_-}^{s'+2,k'_--1,k'_+}$ and satisfies $(P-E)u'=f$. But then we have 
$$
\|u-u'\|_{H_h^{s'+2,k'_+}}\leq C\|u-u'\|_{ \mc{X}_{\Gamma_-}^{s'+2,k'_--1,k'_+}}\leq Ch^{-1}\|(P-E)(u-u')\|_{ \mc{X}_{\Gamma'_-}^{s'+2,k'_-+1,k'_++1}}=0,
$$
so $u=u'$. Now, suppose that there is a sequence $\e_m\to 0^+$ such that $u_m'':=R(E+i\e_m)f$ does not converge to $u$ in $X_{\Gamma_-}^{s+2,k_--1,k_+}$. Then, extracting a subsequence we may assume that $u_m''\to u''\in X_{\Gamma_-}^{s'+2,k'_--1,k'_+}$ and hence $u=u''$, which is a contradiction. In particular,  $R(E+i\e)f\to u$ in $\mc{X}_{\Gamma_-}^{s+2,k_--1,k_+}$ as $\e\to 0^+$. Boundedness of the operator follows from the above estimates. Moreover, we see that if $f\in H_h^{s,k_-}$, for some $k_->\frac{1}{2}$, then $R(E+i0)f\in \mc{X}_{\Gamma_-}^{s+2,k_--1,k_+}$, for any $k_+<-\frac{1}{2}$.

The case of $R(E-i\e)$ follows by an identical argument.
\end{proof}

Finally, we are in a position to prove that the limiting absorption resolvent satisfies the outgoing/incoming property.
\begin{lemma}
For $f\in \mc{E}'(\mathbb{R})$, 
$$
\WF(R(E\pm i0))f\subset \WF(f)\cup \bigcup_{\pm t\geq 0}\exp(tH_{|\xi|^2})\big(\WF(f)\cap \{|\xi|^2=E\}\big).
$$
\end{lemma}
\begin{proof}
First, note that for $A\in \Psi^{0,\comp}$ with $\WF(A)\subset \Ell(B)\cap \Ell(P),$ and 
$$
\|Au\|_{H_h^{k,s}}\leq C\|BPu\|_{H_h^{k-2,s}} +Ch^N\|u\|_{H_h^{-N,-N}}.
$$
Therefore, letting $v_\pm=R(E\pm i0)f$, we have
$$
\WF(v_\pm)\cap\{|\xi|^2\neq E\}\subset\WF(f).
$$
Next, note that since $f\in H_h^{-N,\infty}$, for some $N$, by Lemma~\ref{l:limabs} we have $v_+ \in \mc{X}_{\Gamma_-}^{-N+2,\infty,k}$, and $v_- \in \mc{X}_{\Gamma_+}^{-N+2,\infty,k}$, for any $k<-\frac{1}{2}$ and any $\Gamma_\pm$ open neighborhoods of $L_\pm$ such that $\Gamma_\pm\cap L_\mp=\emptyset$. In particular, by Lemmas~\ref{l:basicProp},~\ref{l:source}, and~\ref{l:sink}, together with the fact that $X_3P\in \Psi^{2,0}$, 
$$
\WFs(R(E\pm i0)f)\cap \{|\xi|^2=E\}\subset \bigcup_{\pm t\geq 0} \exp(tH_p)\Big(\WFs(f)\cap\{|\xi|^2=E\}\Big)\cup L_\pm.
$$
Next, since $f\in\mc{E}'$, $\WFs(f)\subset\{|x|\leq C\}$ and in particular, $\WFs(f)=\WF(f)$. Therefore, the claim follows.
\end{proof}

Using the outgoing property, we can write an effective expression for the incoming/outgoing resolvent (see also~\cite[Lemma 3.60]{DyZw:19})
\begin{lemma}
\label{l:parametrix}
Let $R>0$. Then there is $T>0$ such that for all $f\in \mc{E}'$ supported in $B(0,R)$ and $B\in \Psi^{\comp ,-\infty}$ with $\WFs(B)\subset T^*B(0,R)\cap\{1/2\leq |\xi|\leq 2\}$, and $\chi \in C_c^\infty(B(0,R))$,
$$
\chi R(E\pm i0)B=\frac{i}{h}\int_0^{\pm T} \chi e^{-it(P-E)/h}Bf dt+O(h^\infty)_{\mc{D}'\to C_c^\infty}.
$$
\end{lemma}
\begin{proof}
Let $\psi \in C_c^\infty(\mathbb{R})$ such that $\psi\equiv 1$ on $B(0,R+10T)$. 
Let 
$$v= R(E+i0)Bf-ih^{-1}\int_0^T\psi e^{-it(P-E)/h}Bfdt.$$
 Then, 
\begin{align*}
(P-E)v&=Bf-ih^{-1}\int_0^T(hD_t \psi +[P,\psi])e^{-it(P-E)/h}Bfdt\\
&=\psi e^{-iT(P-E)/h}Bf -ih^{-1}\int_0^T[P,\psi]e^{-it(P-E)/h}Bfdt +O(h^\infty)_{\Psi^{-\infty,-\infty}}f.
\end{align*}
Now, 
$$
\WF(e^{-it(P-E)/h}B)\subset \{(x+2t\xi,\xi)\mid |x|\leq R, |\xi|\in [1/2,2]\}.
$$
In particular, for $t\in[0,T]$, 
$$[P,\psi]e^{-it(P-E)/h}Bf=O(h^\infty)_{C_c^\infty}$$
and we have 
$$
(P-E)v=\psi e^{-iT(P-E)/h}Bf+O(h^\infty)_{C_c^\infty}.
$$
Since $v-R(E+i0)Bf\in C_c^\infty$, for any $\Gamma_-$ a neighborhood of $L_-$ satisfying the assumptions of Lemma~\ref{l:resolve}, $v\in \mc{X}_{\Gamma_-}^{s,k_-,k_+}$ for all $s,k_-$ and $k_+<-\frac{1}{2}$ and hence
$$
v=R(E+i0)(\psi e^{-iT(P-E)/h}Bf+O(h^\infty)_{C_c^\infty}).
$$
But then, 
\begin{align*}
\WF(v)&\subset \{ (x+2t\xi,\xi)\mid t\geq 0,\, (x,\xi)\in \WF(\psi e^{iT(P-E)/h}B)\}\\
&\subset \{(x+2(t+T)\xi,\xi)\mid |x|\leq R,\xi\in[1/2,2]\}\\
&\subset \{(x,\xi)\mid B(0,R+4T)\setminus B(0,T-2R),\,\xi\in[1/2,2]\}.
\end{align*}
In particular, for $T>3R$, $\chi v=O(h^\infty)_{C_c^\infty}$ and hence
$$
\chi R(E+i0)Bf=\frac{i}{h}\int_0^T\chi e^{-it(P-E)/h}Bfdt+O(h^\infty)_{C_c^\infty}
$$
as claimed. The proof for $R(E-i0)$ is identical.

\end{proof}

\section{Completion of the proof of Theorem~\ref{t:main}}
We now complete the proof of the main theorem. Let $P=-h^2\Delta+hW$ where $W$ is admissible (i.e. satisfies~\eqref{e:wForm} and~\eqref{e:estimates}). Let $0<\delta<\delta'<1$. Then by Corollary~\ref{c:gauge}, for any $N>0$, there is $G\in \Psi^0$ self adjoint such that 
$$
P_G:=e^{iG}Pe^{-iG}= -h^2\Delta +hQ+(1-\chi(h^2\Delta-1))h\tilde{W}(1-\chi(h^2\Delta-1))+R_N
$$
where $Q\in \Psi^{-\infty,0}$, $\tilde{W}\in \Psi^1$, are self adjoint, $R_N=O(h^{3N})_{\Psi^{-\infty}}$,  and $\chi \in C_c^\infty $ with $\chi \equiv 1$ on $[-\delta',\delta']$.  In particular, $\tilde{P}_G:=P_G-R_N$ takes the form~\eqref{e:resolveForm}.

Next, note that 
\begin{align*}
\1_{(-\infty,E]}(P)(x,y)&=\langle \1_{(-\infty,E]}(P) \delta_x, 1_{(-\infty,E]}(P)\delta_y\rangle_{L^2}=\langle \1_{(-\infty,E]}(P_G) e^{iG}\delta_x, 1_{(-\infty,E]}(P_G)e^{iG}\delta_y\rangle_{L^2}.
\end{align*}

Now, by~\cite[Lemma 4.2]{PaSh:16}, 
\begin{align*}
\|(\1_{(-\infty, E)}(\tilde{P}_G)-&1_{(\infty,E]}(P_G))f\|_{L^2}\\
&\leq 2\|\1_{[E-\mu,E+\mu]}(\tilde{P}_G)f\|_{L^2}+Ch^{3N}\mu^{-1}(\|\1_{(-\infty,E]}(\tilde{P}_G)f\|_{L^2}+\|(\tilde{P}_G+1)^{-s}f\|_{L^2})
\end{align*}
Let $f\in H^{-\ell}$ and $\mu=h^{N}$. Then for $N,s>\ell$, the last two terms above are bounded by $h^{N}$. Therefore, we need only understand $\1_{(-\infty, E]}(\tilde{P}_G)e^{iG}\delta_x$ and
$
\|\1_{[E-h^{N},E+h^{N}]}(\tilde{P}_G)e^{iG}\delta_x\|_{L^2}.
$

Before, we examine $1_{[a,b]}(\tilde{P}_G)$, we consider the distribution $e^{iG}\delta_x$. By Lemma~\ref{l:eiG}, $e^{iG}\in S^0$, and hence for any $y\in \mathbb{R}$ fixed,
\begin{equation}
\label{e:unitary}
(e^{iG}\delta_y)(x)=\frac{1}{2\pi h}\int e^{\frac{i}{h}(x-y)\xi}b(x,\xi)d\xi
\end{equation}
where $b\in S^0$ with $b\sim \sum_j h^jb_j$, $b_j\in S^{-j}$. In particular, for $|x-y|\geq 1$, and $N>k+1$,
$$
((hD_x)^ke^{iG}\delta_y)(x)=\frac{1}{2\pi h}\int e^{\frac{i}{h}(x-y)\xi}\frac{(-hD_\xi)^N} {|x-y|^{N}}\xi^kb(x,\xi)d\xi = O(|x-y|^{-N}h^{N-1}).
$$
and hence for $\chi \in C_c^\infty$ with $\chi(x) \equiv 1$ on $|x|<R$ and all $|y|<R-1$
$$
(1-\chi)(e^{iG}\delta_y)= O(h^\infty)_{\mc{S}},
$$
where $\mc{S}$ denotes the Schwartz class of functions.
Therefore, 
$$
\1_{[a,b]}(\tilde{P}_G)e^{iG}\delta_x=\1_{[a,b]}(\tilde{P}_G)\chi e^{iG}\delta_x+O(h^\infty)_{C^\infty}.
$$

Next, we consider $\chi \1_{[a,b]}(\tilde{P}_G)\chi$. Let $dE_h$ be the spectral measure for $\tilde{P}_G$. 
\begin{lemma}
\label{l:spectralMeasure}
Let $\chi_1 \in C_c^\infty$ and $\psi\in C_c^\infty$ with $\psi\equiv 1$ on $[-1,1]$. Then, there is $T>0$ such that for $E\in[1-\delta',1+\delta']$, and $h$ small enough,
$$
\chi_1 dE_h\chi_1=\frac{1}{2\pi h} \int_{-T}^T\chi_1 e^{-it(\tilde{P}_G-E)/h}\psi(hD)\chi_1 dt +O(h^\infty)_{\mc{D}'\to C_c^\infty}.
$$
In particular,
$$
\chi_1 dE_h(E)\chi_1=\frac{1}{(2\pi h)^2}\int_{-T}^T\int e^{\frac {i}{h}(-t(|\xi|^2-E)-\langle x-y,\xi\rangle)}a_E(t,x,y,\xi)d\xi dt +O(h^\infty)_{\mc{D}'\to C_c^\infty}
$$
where $a_E\sim\sum_j h^j a_{j,E}$ with $a_{j,E}\in C_c^\infty$. 
\end{lemma}
\begin{proof}
We will use Lemma~\ref{l:parametrix}. In particular, by Stone's formula
$$
\1_{[a,b]}(\tilde{P}_G)=\frac{1}{2\pi i}\int_a^b(R(E+i0)-R(E-i0))dE,
$$
so we need to understand $dE_h:=(2\pi i)^{-1}(R(E+i0)-R(E-i0))$. For this, let $\chi_2 \in C_c^\infty(B(0,R))$ with $\chi_2 \equiv 1$ on $\supp \chi_1$. Then consider
\begin{align*}
\chi dE_h\chi 
&=\frac{1}{2\pi i}\chi_2(R(E+i0)-R(E-i0))(\psi(hD)+1-\psi(hD))\chi_2\\
&=\frac{1}{2\pi h}\int_{-T}^T\chi_2 e^{-it(P-E)/h}\psi(hD)\chi + \frac{1}{2\pi i}\chi_2(R(E+i0)-R(E-i0))(1-\psi(hD))\chi_2
\end{align*}

Let $v_\pm=R(E\pm i0)(1-\psi(hD))\chi_2 f.$ Then, since $(1-\psi(hD))\chi_2 f$ is rapidly decaying, $v_{\pm}$ is semiclassically outgoing/incoming and 
$$
(\tilde{P}_G-E)v_\pm=(1-\psi(hD))\chi_2 f.
$$
In particular, since 
$\WFs ((1-\psi(hD))\chi_2 f)\cap \{p=E\}=\emptyset$, we have
$
\WFs(v_\pm)\cap \{p=E\}=\emptyset.
$

Now, 
$$
(\tilde{P}_G-E)(v_+-v_-)=0\qquad\Rightarrow\qquad 
\WF(v_+-v_-)\setminus \{p=E\}=\emptyset.
$$
In particular, since, a priori both terms have $\WF(v_\pm)\cap \{p=E\}=\emptyset$, we obtain
$$
\WF(v_+-v_-)=\emptyset
$$
and hence
$$
 \frac{1}{2\pi i}\chi_2(R(E+i0)-R(E-i0))(1-\psi(hD))\chi_2 f=O(h^\infty)_{C_c^\infty}.
$$
Therefore, 
$$
\chi_2i dE_h\chi_2=\frac{1}{2\pi h} \int_{-T}^T\chi_2 e^{-it(P-E)/h}\psi(hD)\chi_2 dt +O(h^\infty)_{\mc{D}'\to C_c^\infty}.
$$
The lemma follows from the oscillatory integral formula for $e^{it(P-E)/h}$ (\cite[Theorem 1.4]{Zw:12}).
\end{proof}

As a corollary of Lemma~\ref{l:spectralMeasure}, we obtain for $t,s\in[1-\delta,1+\delta]$,
$$
|(hD_x)^\alpha (hD_y)^\beta \chi_1 \1_{(s,t]}(\tilde{P}_G)\chi_1(x,y)|\leq C_{\alpha\beta}h^{-2}|t-s|.
$$
In particular, this implies
$$
\|\1_{[E-h^{N},E+h^{N}]}(\tilde{P}_G)e^{iG}\delta_x\|_{L^2}\leq Ch^{N-\ell}
$$
for some $\ell>0$ and hence it only remains to have an asymptotic formula for $\chi_1\1_{(-\infty,E]}(\tilde{P}_G)\chi_1$.

Let $\hat{\rho}\in C_c^\infty((-2T,2T))$ with $\hat{\rho} \equiv 1$ on $[-T,T]$ and put $\rho_{h,k}(t)=h^{-k}\rho(th^{-k}).$ Define
\begin{gather}
\label{e:apple}
R_1(E,x,y):=\chi_1(\rho_{h,k}*1_{(-\infty,\cdot]}(\tilde{P}_G)-1_{(-\infty,E]}(\tilde{P}_G))\chi_1(x,y)\\
R_2(E, x,y):=\chi_1(\rho_{h,k}-\rho_{h,1})*1_{(-\infty,\cdot]}(\tilde{P}_G)\chi_1(E,x,y).\label{e:pie}
\end{gather}
Then, we will show for $E\in[1-\delta/2,1+\delta/2]$
\begin{gather}
\label{e:peach}
|(hD_x)^\alpha (hD_y)^\beta R_1(E,x,y)|=O_{\alpha\beta}(h^{k-2}),\qquad |(hD_x)^\alpha (hD_y)^\beta R_2(E,x,y)|=O_{\alpha\beta}(h^{k-2})
\end{gather}

In order to show the first inequality in~\eqref{e:peach} we recall that standard estimates also show that there is $M>0$ such that for $t\in \mathbb{R}$
$$
|(hD_x)^\alpha (hD_y)^\beta \chi_1 \1_{(-\infty,t]}(\tilde{P}_G)\chi_1(x,y)|\leq C_{\alpha \beta}h^{-M}\langle t\rangle^M.
$$
Then, for $E\in[1-\delta/2,1+\delta/2]$
\begin{align*}
&|(hD_x)^\alpha(hD_y)^\beta R_1(E, x,y)|\\
&=\Big|\int h^{-k}\rho(sh^{-k})(hD_x)^\alpha (hD_y)^\beta\chi_1(\1_{(E-s,E]}(\tilde{P}_G))\chi_1ds\Big|\\
&\leq \Big|\int_{|s|\leq \delta/2} h^{-k}\langle sh^{-k}\rangle^{-N} C_\alpha \beta h^{-2}|s|ds\Big|+\Big|\int_{|s|\geq \delta/2} h^{-k}\langle sh^{-k}\rangle^{-N} C_\alpha \beta h^{-M}|s|^Mds\Big|\\
\end{align*}
Choosing $N$ large enough, the first inequality in~\eqref{e:peach} follows.

To obtain the second inequality, we observe that, since $\tilde{P}_G$ is bounded below, 
\begin{align*}
R_2(E)&=\chi_1 \int_{-\infty}^E (h^{-k}(\rho((s-\tilde{P}_G)/h^k)-h^{-1}(\rho((s-\tilde{P}_G)/h))\chi_1\\
&=\frac{1}{2\pi i}\int t^{-1}\hat{\rho}(th^{k-1})(1-\hat{\rho}(t))\chi_1e^{it(E-\tilde{P}_G)/h}\chi_1dt=\chi_1f_{h}\Big(\frac{E-\tilde{P}_G}{h}\Big)\chi_1
\end{align*}
where 
$$
f_h(\lambda)=\frac{1}{2\pi i}\int t^{-1}\hat{\rho}(th^{k-1})(1-\hat{\rho}(t))e^{it\lambda}dt
$$
In particular, note that 
$
|f_h(\lambda)|\leq C_N\langle \lambda \rangle ^{-N}. 
$
Now, let $\psi\in C_c^\infty(-\delta,\delta)$ with $\psi \equiv 1$ near $0$. Then,
\begin{align*}
\chi_1 f_{h}\Big(\frac{E-\tilde{P}_G}{h}\Big)\chi_1&=\int f_h\Big(\frac{E-s}{h}\Big)\chi_1 dE_h(s)\chi_1\\
&=\int\psi(E-s) f_h\Big(\frac{E-s}{h}\Big)\chi_1 dE_h(s)\chi_1 +\int(1-\psi(E-s)) f_h\Big(\frac{E-s}{h}\Big)\chi_1dE_h\chi_1(s)\\
&= \int f_h\Big(\frac{E-s}{h}\Big)\psi(E-s)\chi_1 dE_h(s)\chi_1+O(h^\infty)_{\mc{D}'\to C_c^\infty}\\
&=-\frac{1}{2\pi }\int_{-T}^T\int f_h(w)  \chi_1e^{-it(\tilde{P}_G-E+hw)/h}\psi(hD)\chi_1dwdt+O(h^\infty)_{\mc{D}'\to C_c^\infty}\\
&=\frac{1}{2\pi }\int_{-T}^T\int it^{-1}\hat{\rho}(th^{k-1})(1-\hat{\rho}(t)) \chi_1e^{-it(\tilde{P}_G-E)/h}\psi(hD)\chi_1dt+O(h^\infty)_{\mc{D}'\to C_c^\infty}\\
&=O(h^\infty)_{\mc{D}'\to C_c^\infty}.
\end{align*}
Therefore, the second inequality in~\eqref{e:peach} holds.

Together, the inequalities in~\eqref{e:peach} imply that 
$$
\chi_1 (\1_{(-\infty,E]}(\tilde{P}_G)-\rho_{h,1}*\1_{(-\infty,\cdot]}(\tilde{P}_G)(E))\chi_1=O(h^\infty)_{\mc{D}'\to C_c^\infty}
$$
and we finish the proof of the main theorem by observing that 
\begin{equation}
\begin{aligned}
&\chi_1\rho_{h,1}*\1_{(-\infty,\cdot]}(\tilde{P}_G)(E)\chi_1=\frac{1}{2\pi h}\int_{-\infty}^E \int \hat{\rho}(t)\chi_1e^{it(\mu-\tilde{P}_G)/h}\chi_1dtd\mu \\
&=\frac{1}{(2\pi h)^2}\int_{-\infty}^E \int \hat{\rho}(t)\chi_1(x)e^{i(t(\mu-|\xi|^2)+(x-y)\xi)/h}a(x,y,\xi)\chi_1(y)d\xi dtd\mu
\end{aligned}
\end{equation}
where $a\sim \sum_j a_jh^{j}$ and $a_j\in C_c^\infty$. Conjugating by $e^{iG}$ and using~\eqref{e:unitary} completes the proof.

\appendix
\section{Properties of $s_{k,\semi}$}
\label{a:properties}

In this appendix, we collect the proofs of the required properties of $s_{k,\semi}$.

\begin{proof}[Proof of lemma~\ref{l:basicBound}]
The case $k=1,0$ are clear with $N_0=0$, $N_1=1$. Suppose~\eqref{e:basicBound} holds for $k=n-1$. Then,
$$
s_{n,\semi}(\theta,,\mc{W})=\begin{cases}\frac{1}{|\sum_{i=1}^k\theta_i|}\sum_{p\in Sym(k)}\sum_{|\alpha|=k,\alpha_i\leq k/2}s_{\alpha,\semi}(p(\theta))&\sum_{i=1}^k\theta_i\neq 0\\0&\sum_i\theta_i=0.\end{cases}
$$
The statement is trivial when $\sum_i\theta_i=0$. Therefore, we assume the opposite. In that case
\begin{align*}
&s_{n,\semi}(\theta,\mc{W})\\
&\leq\frac{1}{|\sum_{i=1}^n\theta_i|}\sum_{p\in Sym(n)}\sum_{|\alpha|=n,\alpha_i\leq n/2}\prod_{i=1}^jC_{|\alpha_i|}\frac{\prod_{\ell=1}^{|\alpha_i|}\|w_{p(\theta)_{\beta_i(\alpha)+\ell}}\|_{\semi}}{\inf \{|\omega|^{N_{|\alpha_i|}}\mid\omega\in\{p(\theta)_{\beta_i(\alpha)+1},0\}+\dots+\{p(\theta)_{\beta_{i+1}(\alpha)},0\}\setminus 0\}}\\
&\leq \frac{1}{|\sum_{i=1}^n\theta_i|}\sum_{p\in Sym(k)}\sum_{|\alpha|=n,\alpha_i\leq n/2}\prod_{\ell=1}^{n}\|w_{\theta_\ell}\|_{\semi}\prod_{i=1}^j\frac{C_{|\alpha_i|}}{\inf \{|\omega|^{N_{|\alpha_i|}}\mid\omega\in\{\theta_{1},0\}+\dots+\{\theta_{n},0\}\setminus 0\}}
\end{align*}
Then, defining $N_0=0$, $N_1=1$ and
$$
N_k:=\sup\big\{ 1+\sum_i N_{|\alpha_i|}\mid |\alpha|=n,\,|\alpha_i|\leq \frac{n}{2}\},
$$
we have
$$
s_{n,\semi}(\theta,\mc{W})\leq \frac{\prod_{\ell=1}^{n}\|w_{\theta_\ell}\|_{K}}{\inf \{|\omega|^{N_{k}}\mid\omega\in\{\theta_{1},0\}+\dots+\{\theta_{n},0\}\setminus 0\}}\sum_{p\in Sym(k)}\sum_{|\alpha|=n,\alpha_i\leq n/2}\prod_{i=1}^jC_{|\alpha_i|},
$$
and hence the lemma follows by induction.
\end{proof}

\begin{proof}[Proof of Lemma~\ref{l:induct}]
For $k=0$ the claim is clear. For $k=1$, observe that
$$
s_{1,\semi}(\theta_1+\dots+\theta_n,\tilde{\mc{W}})=\begin{cases}\frac{\|\tilde{w}_{\theta_1\dots\theta_n}\|_{\semi}}{|\sum_{i=1}^n\theta_i|}&\sum_i\theta_i\neq 0\\0&\sum_i\theta_i=0\end{cases}.
$$
Note that 
$$
\frac{\|\tilde{w}_{\theta_1\dots\theta_n}\|_{\semi}}{|\sum_{i=1}^n\theta_i|}\leq \frac{1}{|\sum_{i=1}^n\theta_i|}\prod_{i=1}^n\frac{\|w_{\theta_i}\|_{\semi'}}{|\theta_i|}\leq s_{n,\semi'}((\theta_1,\dots,\theta_n),\mc{W}).
$$

Suppose that the claim holds for $k-1\geq 1$. Then, when $\sum_{i}\sum_{j=1}^k(\theta_i)_j\neq 0$
\begin{align*}
&s_{k,\semi}(\theta_1+\dots +\theta_n,\tilde{\mc{W}})\\
&= \frac{1}{|\sum_{i,j}(\theta_i)_j|}\sum_{p\in Sym(k)}\sum_{|\alpha|=k,\alpha_i\leq k/2}\prod_{i=1}^j s_{\alpha_i,\semi}((p(\theta_1+\dots+\theta_n))_{\alpha,i},\tilde{\mc{W}})\\
&\leq  \frac{1}{|\sum_{i,j}(\theta_i)_j|}\sum_{p\in Sym(k)}\sum_{|\alpha|=k,\alpha_i\leq k/2}\prod_{i=1}^j s_{n\alpha_i,\semi}((p(\theta_1))_{\alpha,i},\dots (p(\theta_n))_{\alpha,i}),\mc{W})\\
&\leq  \frac{1}{|\sum_{i,j}(\theta_i)_j|}\sum_{p\in Sym(nk)}\sum_{|\alpha|=nk,\alpha_i\leq nk/2}\prod_{i=1}^j s_{\alpha_i,\semi}((p(\theta_1,\dots \theta_n))_{\alpha,i}),\mc{W})\\
&=s_{nk,\semi'}(\theta_1,\dots ,\theta_n,\mc{W})\\
\end{align*}

\vspace{-1.3cm}
\end{proof}

\section{Examples with infinitely many embedded eigenvalues}
\label{a:example}

We now construct some examples to which our main theorem applies that, nevertheless, have arbitrarily large eigenvalues.

\begin{theorem}
\label{t:ex1}
Let $\omega\in \mathbb{R}^d$ satisfy the diophantine condition~\eqref{e:diophantine} and $\Theta=\mathbb{Z}^d\cdot \omega$.  Then there is $W\in C^\infty(\mathbb{R};\mathbb{R})$ satisfying the assumptions of Theorem~\ref{t:almostPeriodic} and such that 
$
\{\tfrac{\theta^2}{4}\mid \theta\in \Theta\setminus \{0\}\}
$
is contained in the point spectrum of $-\Delta+W$. 
\end{theorem}

\begin{theorem}
\label{t:ex2}
Let $\{m_n\}_{n=1}^\infty\subset \mathbb{Z}_+$ and $\Theta$ as in Theorem~\ref{t:limitPeriodic}. Then there is $W\in C^\infty(\mathbb{R};\mathbb{R})$ satisfying the assumptions of Theorem~\ref{t:limitPeriodic} and such that for all $n$, $\frac{m_n^2}{4n^2}$ is contained in the point spectrum of $-\Delta+W$. In particular, if $\mathbb{Q}\cap \mathbb{R}_+= \{\frac{m_n}{n}\}_{n=1}^\infty$, then this operator has dense pure point spectrum.
\end{theorem}

Theorems~\ref{t:ex1} and~\ref{t:ex2} follow easily from the following theorem.
\begin{theorem}
\label{t:eigenvalue}
Let $\{\kappa_n\}_{n=1}^\infty$ be an arbitrary sequence of positive real numbers. Then there is $W\in C^\infty(\mathbb{R};\mathbb{R})$ such that $\kappa_n^2$ is an eigenvalue of $-\Delta +W$. Moreover, we can find $W$ such that 
$$
W=\sum_n e^{2i\kappa_n x}w_{2\kappa_n}(x)+\sum_n e^{-2i\kappa_n x}w_{-2\kappa_n}(x) + w_0(x)
$$
where $w_0\in C_c^\infty$ and for any $N$,
\begin{equation}
\label{e:rapidDecay}
|\partial_x^k w_{\pm 2\kappa_n}(x)|\leq C_N\langle n\rangle^{-N}\langle \kappa_n\rangle^{-N}\langle x\rangle^{-k}.
\end{equation}
\end{theorem}

We follow the construction in~\cite{Si:97} with a few modifications to guarantee smoothness. First, we need to replace ~\cite[Theorem 5]{Si:97} to allow for smoothness in $V$.

Recall that the Pr\"ufer angles $,\phi(x)$, are defined by 
$$
u'(x)=k A(x)\cos(\phi(x)),\qquad u(x)=A(x)\sin(\phi(x))
$$
where $-u''+V(x)u=k^2u$. Then, $\phi(x)$ satisfies
\begin{equation}
\label{e:prufer}
\phi'(x)=k-k^{-1}V(x)\sin^2(\phi(x)).
\end{equation}
For any $N\geq 0$, $a<b\in\mathbb{R}$. let $F:C^N([a,b])\times \mathbb{R}^n\times \mathbb{T}^n\to \mathbb{T}^n$ to be the generalized Pr\"ufer angles with potential $V$, $\phi_i(x;V,k,\theta)|_{x=b}$, where $\phi_i(0;V,k,\theta)=\theta_i$ and we put $k=k_i$ in~\eqref{e:prufer}.

\begin{lemma}
\label{l:rotate}
Fix $[a,b]\subset (0,\infty)$, $U\Subset (a,b)$ open, $N>0$, $k_1,\dots k_n>0$ distinct,  $\theta^{(0)}\in \mathbb{T}^n$, and $\e>0$. Then there is $\delta>0$ such that for all angles $\theta^{(1)}\in \mathbb{T}^n$ satisfying
$$
|\theta^{(1)}- kb-\theta^{(0)}|<\delta,
$$
there is $V\in C_c^\infty(U)$ with $\|V\|_{C^N}<\e$ and $F(V,k,\theta^{(0)})=\theta^{(1)}.$
\end{lemma}
\begin{proof}
Note that $F(0,k,\theta^0)=(\theta_1^{(0)}+k_1b,\dots \theta_n^{(0)}+k_nb)$ and 
$
\phi_i(x;V=0)=\theta_i^{(0)}+k_ix.
$
Therefore, we need only show that the differential (in $V$) is surjective when restricted to functions in $C_c^\infty(U)$. For this, let $\chi \in C_c^\infty(U)$ with $\chi \equiv 1$ on a nonempty open interval $I$. Note that if $V_\e=\e \chi V(x)$, 
$$
\partial_\e \phi_i'(x;V_\e)|_{\e=0}=-k_i^{-1}\chi(x)V(x)\sin^2(k_ix+\theta_i^{(0)}),\qquad \partial_\e \phi_i(0;V_\e)|_{\e=0}=0.
$$
Hence,
$$
\partial_\e F_i(V_\e)|_{\e=0}=-k_i^{-1}\int \chi(x)V(x)\sin^2(k_i x+\theta_i^{(0)})dx.
$$
We claim that $u_i(x):=\chi(x)\sin^2(k_i x+\theta_i^{(0)})$ are linearly independent in $L^2$. Indeed, suppose $0<k_1<\dots k_n$ and $\sum_{i=1}^K \alpha_i u_i(x)=0$  a.e. with $\alpha_K\neq 0$ (and hence, by continuity for all $x$). Differentiating enough times, we see that $\alpha_K\equiv 0$, a contradiction.

Thus, there are $V_1,\dots V_n\in C^\infty$ such that $(\partial_\e F(\e \chi V_i))_{i=1}^n$ is a basis for $\mathbb{R}^n$ and the implicit function theorem finishes the proof.
\end{proof}

\begin{proof}[Proof of Theorem~\ref{t:eigenvalue}]
We work on the half line and find $W(x)$ vanishing to infinite order at $0$ such that there are $L^2$ solutions, $u_n$ of 
$$
-u_n''(x)+W(x)u_n(x)=\kappa_n^2u_n(x),\quad x\in[0,\infty) \qquad u_n(0)=0.
$$
The case of the line then follows by extending $W$ to an even function and $u_n$ to an odd function.

Let $\chi \in C^\infty(\mathbb{R})$ with $\chi \equiv 1$ on $[2,\infty)$, $\supp \chi \subset (1,\infty)$ and define $\chi_n(x):=\chi(R_n^{-1}x)$ where $R_n\to \infty$, $R_n\geq 1$ are to be chosen later. We put
$$
(\Delta L_n)(x):=4 \kappa_n \frac{\chi_n(x)}{x}\sin (2\kappa_n x+\varphi_n)
$$
where $\varphi_n$ is also to be chosen. We will also find $\Delta S_n$ to be smooth function supported on $(2^{-n},2^{-n+1})$ with $\|\Delta S_n\|_{C^n}\leq \frac{1}{2^n}$ and put
$$
W_m(x)=\sum_{n=1}^m(\Delta L_n+\Delta S_n)(x),\qquad W(x):=\lim_{m\to \infty} W_m(x),\qquad \tilde{W}_m:=W_m-\Delta S_m.
$$

Note that by construction $\sum_n \Delta S_n\in C^\infty([0,1))$, $\sum_n\Delta S_n$ vanishes to infinite order at $0$, and 
$$
\Delta L_n(x)= -e^{2i\kappa_n x}2i\kappa_n e^{i\varphi_n}\chi_n(x)x^{-1}+e^{-2i\kappa_n x}2i\kappa_n e^{-i\varphi_n}\chi_n(x)x^{-1}.
$$
In particular, 
$$
\Delta L_n(x)= e^{2i\kappa_n x}w_{2\kappa_n}(x)+e^{-2i\kappa_n x}w_{-2\kappa_n}(x)
$$
with $w_{\pm 2\kappa_n}= \mp 2i \kappa_n e^{\pm i\varphi_n}\chi_n(x)x^{-1}.$ In particular,
\begin{equation}
\label{e:symbolW}
|\partial_x^k w_{\pm 2\kappa_n}|\leq C_k\kappa_n R_n^{-1}\langle x\rangle^{-k}.
\end{equation}
In order to obtain the estimate~\eqref{e:rapidDecay}, we fix a positive Schwartz function $f$ and choose $R_n$ $R_n\geq  \frac{1}{f(\langle \kappa_n\rangle\langle n\rangle)}$. The estimate~\eqref{e:symbolW} then guarantees that $\sum_n \Delta L_n$ is bounded with all derivatives. The fact that $\Delta S_n \in C_c^\infty(2^{-n},2^{-n+1})$ and $\|\Delta S_n\|_{C^n}\leq \frac{1}{2^n}$ guarantees that $w_0=\sum_n \Delta S_n\in C^\infty([0,1))$ and $w_0$ vanishes to infinite order at $0$.

Now, note that $\phi_n(\xi):=\mc{F}((\cdot)^{-1}\chi_n(\cdot))(\xi)$ is smooth away from $\xi=0$. Therefore, for each $m\neq n$, we can find $\psi_{n,m}\in C_c^\infty( 0,1)$ such that 
\begin{gather*}
\mc{F}(\psi_{n.m})(0)=-2i\kappa_n(-\phi_n(2\kappa_n)e^{i\varphi_n}-\phi_n(2\kappa_n)e^{-i\varphi_n})\\
\mc{F}(\psi_{n,m})(\pm 2\kappa_m)=-2i\kappa_n(\phi_m(2(\pm \kappa_m-\kappa_n))e^{i\varphi_n}-\phi_n(2(\kappa_n\pm\kappa_m))e^{-i\varphi_n}).
\end{gather*}

Then, letting $\psi_{n,n}=0$ and defining $\tilde{L}_{n,m}:=\Delta L_n-\psi_{n,m}$, there are $A_{n,m}$, $A_{n,m}^{\pm}$ such that 
\begin{equation}
\label{e:condition}
\begin{gathered}
|\tilde{L}_{n,m}|\leq C|x|^{-1},\qquad \tilde{L}_{n,m}=A_{n.m}',\quad |A_{n,m}|\leq C|x|^{-1},\\
e^{\pm 2i\kappa_m x}\tilde{L}_{n,m}=(A_{n,m}^\pm)',\quad |A_{n,m}^\pm(x)|\leq C|x|^{-1}.
\end{gathered}
\end{equation}

By the conditions~\eqref{e:condition} and~\cite[Theorem 3]{Si:97}, there is a unique function $u_n^{(m)}(x)$ satisfying
\begin{gather}
-(u_n^{(m)})''+W_m(x)u_n^{(m)}=\kappa_n^2u_{n}^{(m)},\qquad \N{u_n^{(m)}-\sin((\kappa_n+\tfrac{1}{2}\varphi_n)\cdot)(1+|\cdot|)^{-1}}<\infty.
\end{gather}
where $\N{u}=\|(1+x^2)u\|_\infty+\|(1+x^2)u'\|_\infty.$ Similarly, there is a unique function $\tilde{u}_{n}^{(m)}(x)$ satisfying
\begin{gather}
\label{e:asymptoticConstruct}
-(\tilde{u}_n^{(m)})''+\tilde{W}_m(x)\tilde{u}_n^{(m)}=\kappa_n^2\tilde{u}_{n}^{(m)},\qquad \N{\tilde{u}_n^{(m)}-\sin((\kappa_n+\tfrac{1}{2}\varphi_n)\cdot)(1+|\cdot|)^{-1}}<\infty.
\end{gather}

Now, we construct $\Delta L_n$, $\Delta S_n$ such that 
\begin{gather}
\label{e:perturbIt}
\N{u_n^{(m)}-u_{n}^{(m-1)}}\leq 2^{-m},\quad n=1,2,\dots, m-1,\qquad u_n^{(m)}(0)=0,\quad n=1,\dots,m.
\end{gather}
Once we have done this, we can let $u_n=\lim_{m}u_n^{(m)}$ (in the $\N{\cdot}$ norm) to obtain $L^2$ eigenfunctions with eigenvalue $\kappa_n$.

Let $m\geq 1$ and suppose we have chosen $\{(R_n,\varphi_n)\}_{n=1}^{m-1}$, and $\Delta S_1,\dots \Delta S_{m-1}\in C_c^\infty$ with $\supp \Delta S_n\subset (2^{-n},2^{-n+1})$ and $\|\Delta S_n\|_{C^n}\leq \frac{1}{2^n}$ such that~\eqref{e:perturbIt} holds and $R_n\geq 1/f(\langle n\rangle \langle \kappa_n\rangle)$.

By~\cite[Theorem 3]{Si:97}, there are $\e_m$ $\tilde{R}_m$ such that for all $R_m\geq \tilde{R}_m$, and $\varphi_m\in[0, 2\pi/(2\kappa_m)]$,  if $\|\Delta S_m\|_{C^0}\leq \e_m$, then
$$
\N{u_i^{m}-\tilde{u}_i^m}\leq 2^{-m-1}.
$$
Observe that by Lemma~\ref{l:rotate}, there is $\delta_m>0$ small enough such that if $|\theta_i^{(1)}-\kappa_i 2^{-m+1}|<\delta_m$ and $\theta_i^{(1)}$ are the Pr\"ufer angles of the solutions $\tilde{u}_i^{m}$, $i=1,\dots, m$ at $2^{-m+1}$, then there is $\Delta S_m\in C_c^\infty( 2^{-m},2^{-m+1})$ with $\|\Delta S_m\|_{C^m}\leq\min(2^{-m},\e_m)$ and such that $u_i^{(m)}(0)=0$.  Therefore, if we can find $R_m\geq \tilde{R}_m$ and $\varphi_m$ such that $|\theta_i^{(1)}-\kappa_i2^{-m+1}|<\delta_m$, and
$$
\N{u_i^{m-1}-\tilde{u}_i^m}\leq 2^{-m-1},
$$
the proof will be complete.

Once again by~\cite[Theorem 3]{Si:97}, for $R_m$ large enough, we have (uniformly in $\varphi_m\in[0,2\pi/(2\kappa_m)])$, $\N{u_i^{(m-1)}-\tilde{u}_i^{(m)}}<2^{-m-1}$ for $i=1,\dots, m-1$ and the Pr\"ufer angles for $\tilde{u}_i^{(m)}$ at $2^{-m+1}$ satisfy $|\theta_i^{(1)}-\kappa_i b_i|<\delta$ for $i=1,\dots m-1$). 

Finally, we choose $\varphi_m$ so that $\tilde{u}_m^{(m)}(0)=0$. The existence of such a $\varphi_m$ again follows from~\cite[Theorem 3]{Si:97}. In particular, note that by part (b) there, we have~\eqref{e:asymptoticConstruct} uniformly over $R_m$ large enough, $x$ large enough, and $\varphi_m\in[0,2\pi/(2\kappa_m)]$. In particular, the Pr\"ufer angles for $\tilde{u}_m^{(m)}$, $\tilde{\phi}_m(x)$  run through a full circle. Therefore, we can choose $R_m$ large enough and $\varphi_m$ such that the $\tilde{\phi}_m(R_n)$ agrees with the Pr\"ufer angle of the solution to $u$ to $-u''+W_{m-1}(x)u=\kappa_m^2u$, $u(0)=0$ and hence, since $W_{m-1}=\tilde{W}_m$ on $x\leq R_n$, we have that $\tilde{u}_m(0)=0$.

\end{proof}

\bibliography{biblio}
\bibliographystyle{alpha}

\end{document}